\documentclass[12pt,reqno]{amsart}
\usepackage{mathrsfs}
\usepackage{array}
\usepackage{cases}
\usepackage{epic}
\usepackage{amsfonts}
\usepackage{graphicx}
\usepackage{amsmath}
\usepackage{amssymb, upgreek}
\usepackage{bm}
\usepackage{latexsym,todonotes}
\usepackage{pdflscape}
\usepackage[all]{xypic}
\usepackage[all]{xy}
\usepackage{color}
\usepackage{colordvi}
\usepackage{multicol}
\usepackage[normalem]{ulem}
\usepackage[mathscr]{euscript}
\input diagxy
\usepackage[linktocpage=true]{hyperref}
\hypersetup{colorlinks,linkcolor=blue,urlcolor=cyan,citecolor=blue}
\usepackage{chngcntr} 
\counterwithin{figure}{section}
\counterwithin{table}{section}
\usetikzlibrary{
	cd,
	shapes,
	arrows,
	positioning,
	decorations.markings,
	decorations.pathmorphing,
	circuits.logic.US,
	circuits.logic.IEC,
	fit,
	calc,
	plotmarks,
	matrix
}
\begin{document}
	\input xy
	\xyoption{all}

	%Theorem for the introduciton
	\newtheorem{innercustomthm}{{\bf Theorem}}
	\newenvironment{customthm}[1]
	{\renewcommand\theinnercustomthm{#1}\innercustomthm}
	{\endinnercustomthm}
	
	\newtheorem{innercustomcor}{{\bf Corollary}}
	\newenvironment{customcor}[1]
	{\renewcommand\theinnercustomcor{#1}\innercustomcor}
	{\endinnercustomthm}
	
	\newtheorem{innercustomprop}{{\bf Proposition}}
	\newenvironment{customprop}[1]
	{\renewcommand\theinnercustomprop{#1}\innercustomprop}
	{\endinnercustomthm}

\newtheorem{alphatheorem}{Theorem}
\newtheorem{alphacorollary}[alphatheorem]{Corollary}
\newtheorem{alphaproposition}[alphatheorem]{Proposition}
\renewcommand*{\thealphatheorem}{\Alph{alphatheorem}}	
	
	\newtheorem{theorem}{Theorem}[section]
	\newtheorem{acknowledgement}[theorem]{Acknowledgement}
	\newtheorem{algorithm}[theorem]{Algorithm}
	\newtheorem{axiom}[theorem]{Axiom}
	\newtheorem{case}[theorem]{Case}
	\newtheorem{claim}[theorem]{Claim}
	\newtheorem{conclusion}[theorem]{Conclusion}
	\newtheorem{condition}[theorem]{Condition}
	\newtheorem{conjecture}[theorem]{Conjecture}
	\newtheorem{construction}[theorem]{Construction}
	\newtheorem{corollary}[theorem]{Corollary}
	\newtheorem{Claim}[theorem]{Claim}
	\newtheorem{definition}[theorem]{Definition}
	\newtheorem{example}[theorem]{Example}
	\newtheorem{exercise}[theorem]{Exercise}
	\newtheorem{lemma}[theorem]{Lemma}
	\newtheorem{notation}[theorem]{Notation}
	\newtheorem{problem}[theorem]{Problem}
	\newtheorem{proposition}[theorem]{Proposition}
	\newtheorem{solution}[theorem]{Solution}
	\newtheorem{summary}[theorem]{Summary}
	\numberwithin{equation}{section}
	
	\theoremstyle{remark}
	\newtheorem{remark}[theorem]{Remark}

	%%%%%%%%%%%%%%%%%%%%%%%%%%%%%%%%%%%%%%%%%%%%%%%%%%%%%%%%%%%%
	\makeatletter
	\newcommand{\rmnum}[1]{\romannumeral #1}
	\newcommand{\Rmnum}[1]{\expandafter\@slowromancap\romannumeral #1@}
	\def \g{\mathfrak{g}}
	\def \bH{{\mathbf H}}
	\def \bB{{\mathbf B }}
	\def \BKH{{\acute{H}}}
	\def \bBKH{{\acute{\mathbf H}}}
	\def \bK{\mathbb{K}}
	\def \bC{\mathbb{\bK_\de}}
	\def \bQ{\mathbb{Q}}
	\def \Z{\mathbb{Z}}
	\def \I{\mathbb{I}}
	\def \ch{{\mathcal H}}
	\def \cm{{\mathcal M}}
	\def \ct{{\mathcal T}}
	\def \hW{\widehat{W}}
	\def \bZ{\mathbb{Z}}
	\def \TH{\Theta}
	
	\def \tMH{{\cm\widetilde{\ch}(\Lambda^\imath)}}
	\def \tMHl{{\cm\widetilde{\ch}(\bs_\ell\Lambda^\imath)}}
	\def \bTH { \boldsymbol{\Theta}}
	\def \bDel{ \boldsymbol{\Delta}}
	\def \tY{\widetilde{Y}}
	\def \tH{{\widetilde{H}}}
	\def \tTH{\widetilde{\Theta}}
	\def \btau{\tau}
	\def \s{\varsigma}
	\def \bvs{\boldsymbol{\varsigma}}
	\def \bs{\mathbf r}
	\def \cR{\mathcal{R}}
	\def \fg{\mathfrak{g}}
	\def \Aut{\operatorname{Aut}\nolimits}
	
	%%WW
	\newcommand{\hgt}{\text{ht}}
	\def \ad{\text{ad}\,}
	\def \gr{\text{gr}\,}
	\def \hg{\widehat{\g}}
	\def \mf{\mathfrak}
	\def \N{\mathbb N}
	\def \tk{\widetilde{k}}
	\def \brW{{\rm Br}(W_{\tau})}
	\def \bome{{\varpi}}
	\def \bth{\bm{\theta}}
	\def\Br{{\rm Br}}
	\renewcommand{\t}{\boldsymbol{\omega}}
	
	%%WZ
	\def \tfX{\widetilde{ \Upsilon}}
	\def \tT{\widetilde{T}}

	\def \tcT{\widetilde{\mathscr T}}
	\def \E{\widetilde{E}}
	
	\def \SS{\mathbb{S}}
	\def \sll{\mathfrak{sl}}
	\def \tUD{{}^{Dr}\tU}

	\newcommand{\UU}{{\mathbf U}\otimes {\mathbf U}}
	\newcommand{\UUi}{(\UU)^\imath}
	\newcommand{\tUU}{{\tU}\otimes {\tU}}
	\newcommand{\tUUi}{(\tUU)^\imath}
	\newcommand{\tK}{\widetilde{K}}
	\newcommand{\tU}{\widetilde{{\mathbf U}} }
	\newcommand{\tUi}{\widetilde{{\mathbf U}}^\imath}
	\newcommand{\tUiii}{\tUi(\widehat\sll_3,\tau)}
	\newcommand{\tUiD}{{}^{\mathrm{Dr}}\widetilde{{\mathbf U}}^\imath}
	\newcommand{\tUigr}{{\text{gr}}\tUi}
	\newcommand{\tUigrp}{{\text{gr}}\widetilde{{\mathbf U}}^{\imath,+}}
	\newcommand{\tUiDgr}{{\text{gr}}\tUiD}

	\newcommand{\sqq}{{\bf v}}
	\newcommand{\sqvs}{\sqrt{\vs}}
	\newcommand{\dbl}{\operatorname{dbl}\nolimits}
	\newcommand{\swa}{\operatorname{swap}\nolimits}
	\renewcommand{\Im}{\operatorname{Im}\nolimits}
	\newcommand{\Sym}{\operatorname{Sym}\nolimits}
	\newcommand{\qbinom}[2]{\begin{bmatrix} #1\\#2 \end{bmatrix} }
	
	\def \ov{\overline}
	\def \balpha{{\boldsymbol \alpha}} %{\ov{\alpha}}
	\def \K{\mathbb K}
	
	\newcommand{\nc}{\newcommand}
	%\nc{\browntext}[1]{\textcolor{brown}{#1}}
	\nc{\browntext}[1]{\textcolor{brown}{#1}}
	\nc{\greentext}[1]{\textcolor{green}{#1}}
	\nc{\redtext}[1]{\textcolor{red}{#1}}
	\nc{\bluetext}[1]{\textcolor{blue}{#1}}
	\nc{\brown}[1]{\browntext{ #1}}
	\nc{\green}[1]{\greentext{ #1}}
	\nc{\red}[1]{\redtext{#1}}
	\nc{\blue}[1]{\bluetext{#1}}

	\def \Q {\mathbb Q}
	\def \C{\mathbb C}
	\def \TT{\mathbf T}
	\def \tt{\mathbf t}
	\newcommand{\wt}{\text{wt}}
	\def \de{\delta}
	\def \bvs{{\boldsymbol{\varsigma}}}
	\def \vs{\varsigma}
	\def \U{\mathbf U}
	\def \Ui{\mathbf{U}^\imath}
	\def \dm{\diamond}
	\def \bS{\mathbb S}
	
	\def \bR{\mathbb R}
	\def \bP{\mathbb P}
	\def \bF{\mathbb F}
	\def \II{\mathbb{I}_0}
	\def \t{\Omega}
	\allowdisplaybreaks
	
	\newcommand{\lutodo}{\todo[inline,color=violet!20, caption={}]}
	\newcommand{\wtodo}{\todo[inline,color=cyan!20, caption={}]}
	\newcommand{\ztodo}{\todo[inline,color=green!20, caption={}]}
	
	%%%%%
	\title[Quasi-split affine $\mathrm{i}$quantum groups III]{Braid group action and \\quasi-split affine $\mathrm{i}$quantum groups III}
	
	\author[Ming Lu]{Ming Lu}
	\address{Department of Mathematics, Sichuan University, Chengdu 610064, P.R.China}
	\email{luming@scu.edu.cn}

	\author[Xiaolong Pan]{Xiaolong Pan}
	\address{Department of Mathematics, Sichuan University, Chengdu 610064, P.R.China}
	\email{xiaolong\_pan@stu.scu.edu.cn}
	
	\author[Weiqiang Wang]{Weiqiang Wang}
	\address{Department of Mathematics, University of Virginia, Charlottesville, VA 22904, USA}
	\email{ww9c@virginia.edu}

	\author[Weinan Zhang]{Weinan Zhang}
	\address{Department of Mathematics and New Cornerstone Science Laboratory, The University of Hong Kong, Pokfulam, Hong Kong SAR, P.R.China}
	\email{mathzwn@hku.hk}

	\subjclass[2010]{Primary 17B37.}
	\keywords{Affine iquantum groups, Braid group action, Quantum symmetric pairs, Drinfeld presentation}

\begin{abstract}
	This is the last of three papers on Drinfeld presentations of quasi-split affine iquantum groups $\widetilde{{\mathbf U}}^\imath$, settling the remaining type ${\rm AIII}^{(\tau)}_{2r}$. This type distinguishes itself among all quasi-split affine types in having 3 relative root lengths. Various basic real and imaginary $v$-root vectors for $\widetilde{{\mathbf U}}^\imath$ are constructed, giving rise to affine rank one subalgebras of $\widetilde{{\mathbf U}}^\imath$ associated with simple roots in the finite relative root system. We establish the relations among these $v$-root vectors and show that they provide a Drinfeld presentation of $\widetilde{{\mathbf U}}^\imath$.
\end{abstract}
	
	\maketitle
    \setcounter{tocdepth}{1}
	\tableofcontents

%%%%%%%
%%%%%%%
\section{Introduction}

In this sequel to \cite{LWZ23, LWZ24}, we shall give the Drinfeld presentation for the affine iquantum group of type AIII$_{2r}^{(\tau)}$ with the following Satake diagram $(\I,\tau)$: 
\begin{center}\setlength{\unitlength}{0.7mm}
	\begin{equation}
		\label{pic:AIII2r}
	\begin{picture}(40,14)(0,20)
	\put(-70,20){${\rm AIII}_{2r}^{(\tau)}\, (r\ge 1)$}
				\put(0,10){$\circ$}
				\put(0,30){$\circ$}
				\put(50,10){$\circ$}
				\put(50,30){$\circ$}
				\put(72,10){$\circ$}
				\put(72,30){$\circ$}
				\put(73.5,13){\line(0,1){17}}
				\put(0,5.5){$2r$}
				\put(-2,34){${1}$}
				\put(47,6){\small $r+2$}
				\put(47,34){\small $r-1$}
				\put(69,6){\small $r+1$}
				\put(69,34){\small $r$}
				
				\put(3,11.5){\line(1,0){16}}
				\put(3,31.5){\line(1,0){16}}
				\put(23,10){$\cdots$}
				\put(23,30){$\cdots$}
				\put(33.5,11.5){\line(1,0){16}}
				\put(33.5,31.5){\line(1,0){16}}
				\put(53,11.5){\line(1,0){18.5}}
				\put(53,31.5){\line(1,0){18.5}}
				\put(-17.5,22.5){\line(2,1){17}}
				\put(-17.5,20.5){\line(2,-1){17}}
				
				\put(-21,20){$\circ$}
				\put(-21,15){\small $0$}
				
				\color{red}
				\put(0,20){\small $\tau$}
				\put(50,20){\small $\tau$}
				\put(78,20){\small $\tau$}

                \qbezier(0,13.5)(-4,21.5)(0,29.5)
				\put(-0.25,14){\vector(1,-2){0.5}}
				\put(-0.25,29){\vector(1,2){0.5}}
				
				\qbezier(50,13.5)(46,21.5)(50,29.5)
				\put(49.75,14){\vector(1,-2){0.5}}
				\put(49.75,29){\vector(1,2){0.5}}

				\qbezier(75,13.5)(79,21.5)(75,29.5)
				\put(75.25,14){\vector(-1,-2){0.5}}
				\put(75.25,29){\vector(-1,2){0.5}}
			\end{picture}
		\end{equation}     
		\vspace{0.5cm}
	\end{center}
This completes the constructions of Drinfeld presentations for all quasi-split affine iquantum groups. Drinfeld presentations of affine iquantum groups were obtained first in \cite{LW21b} for split ADE type and then in \cite{Z22} for split BCFG type. 
\vspace{2mm}

Drinfeld's new presentation \cite{Dr87} (established by Beck \cite{Be94} and Damiani \cite{Da15}) exhibits a quantized loop algebra structure for Drinfeld-Jimbo quantum groups of affine type. They have played an important role in representation theory of affine quantum groups, $q$-characters, quantum integrable models (such as XXZ-models and variants), and admit connections to cluster algebras and monoidal categorification. See the ICM lecture of Hernandez \cite{Her25} for a recent survey and references therein. 

We view quasi-split iquantum groups arising from quantum symmetric pairs (introduced by Letzter \cite{Let99} and generalized by Kolb \cite{Ko14}) as a generalization of Drinfeld-Jimbo quantum groups. In this paper we shall always work with universal iquantum groups following \cite{LW22}; this version of iquantum groups naturally affords  ibraid (= relative braid) group symmetries \cite{KP11, LW21a, LW22b, WZ23, Z23} which are essential for the constructions of Drinfeld presentations. 
Affine iquantum groups and their degenerate cousin known as twisted Yangians are closely related to quantum integrable models with boundary conditions (cf. Sklyanin \cite{Skl88}) and boundary affine Toda field theories (see Baseilhac-Belliard \cite{BB10}). 

Drinfeld type presentations of affine iquantum groups \cite{LW21b, Z22, LWZ23, LWZ24} (see also Baseilhac-Kolb \cite{BK20} for q-Onsager algebra) exhibit a quantized twisted loop algebra structure. The quasi-split affine Satake diagrams $(\I =\I_0\cup\{0\},\tau)$ relevant to Drinfeld presentations of affine iquantum groups are always assumed to satisfy that $\tau$ fixes the affine node $0$, and hence they are in bijection with finite  Satake diagrams $(\I_0, \tau)$. There are 3 rank one (quasi-split) Satake diagrams. Accordingly there are 3 distinct affine iquantum groups of rank one: besides the split rank one known as q-Onsager algebra, we have 
	\begin{enumerate}
		\item[(\texttt{A})] Drinfeld-Jimbo affine quantum $\sll_2$;
		\item[(\texttt{B})] affine iquantum group $\tUi(\widehat{\mathfrak{sl}}_3,\tau)$ of type AIII$_{2}^{(\tau)}$. 
        %, where $\tau$ is the involution on $\I =\{0,1,2\}$ of type $\widehat A_2$, such that	$\tau(0) =0$ and $1 \stackrel{\tau}{\leftrightarrow} 2.$
	\end{enumerate}
The Drinfeld presentations for (\texttt{A}) and (\texttt{B}) obtained by Damiani \cite{Da93} and the authors \cite{LWZ23} will play a fundamental role in the higher rank case studied in this paper. 

The affine iquantum groups of type AIII$_{2r}^{(\tau)}$, with Satake diagram \eqref{pic:AIII2r}  and denoted by $\tUi$, are the only family of affine iquantum groups whose Drinfeld presentation remains open (for $r\ge 2$) and requires the affine rank one in (\texttt{B}) as a new building block; the goal of this paper is to settle this last case. This $\tUi$ distinguishes itself among all affine types in having root vectors of 3 distinct root lengths. Moreover, $\tUi$ admits an action of the ibraid group $\text{Br}(W_\tau)$ of type $A_{2r}^{(2)}$ (cf. \cite{LWZ23, WZ23, Z23}) with its relative root system given below:
\begin{table}[h]  
	\begin{center}
		\centering
	\begin{tabular}{|m{4cm}<{\centering}%|m{4cm}<{\centering}
    |m{10cm}<{\centering}|}
	\hline
	Satake types & Relative affine root systems \\
	\hline
	\begin{tikzpicture}[baseline=0, scale=2.5]
    \node at (0, -0.02) {${\rm AIII}_{2}^{(\tau)}$ };
	\end{tikzpicture}  
    & \begin{tikzpicture}[baseline=0, scale=2.5]
    \node at (-1.0, -0.02) {${\rm A}_2^{(2)}$ };
	\node at (-0.5,-0.02) {\large$\circ$};
	\node at (0,-0.02) {\large$\circ$};
					
	\draw[-](-0.45, 0.025) to (-0.07, 0.025);
	\draw[-](-0.45, 0) to (-0.06, 0);
	\draw[-](-0.45, -0.025) to (-0.06, -0.025);
	\draw[-](-0.45, -0.05) to (-0.07, -0.05);
	\node at (0, -.12) {\small $1$};
	\draw[-](-0.08,-0.07) to (-0.05,-0.0125) to (-0.08,0.045);
	\node at (-0.5,-.12) {\small $0$};
	\end{tikzpicture} 
		\\
	\hline
	\begin{tikzpicture}[baseline=0, scale=2.5]
    \node at (0, 0.0) {${\rm AIII}_{2r}^{(\tau)}$ ($r\geq 2$) };
	\end{tikzpicture}  	 
    & 	\begin{tikzpicture}[baseline=0, scale=2.5]
    \node at (-1.0, 0.0) {${\rm A}_{2r}^{(2)}$};
	\node at (-0.5,0) {$\circ$};
	\node at (0,0) {$\circ$};
	\node at (0.5,0) {$\circ$};
	\node at (1.5,0) {$\circ$};
	\node at (2.0,0) {$\circ$};
	\node at (2.5,0) {$\circ$};
	\draw[-] (0.05,0) to (0.45,0);
	\draw[-] (1.55,0) to (1.95,0);
	\draw[dashed] (0.55,0) to (1.45,0);
	\draw[-implies, double equal sign distance]  (2.05, 0) to (2.45, 0);
	\draw[-implies, double equal sign distance]  (-0.45, 0) to (-0.05, 0);
	\node at (2, -.2) {\small $r-1$};
	\node at (2.5,-.2) {\small $r$};
	\node at (0, -.2) {\small $1$};
	\node at (0.5,-.2) {\small $2$};
	\node at (-0.5,-.2) {\small $0$};
\end{tikzpicture}
\\\hline
\end{tabular} 
\end{center}
	\vspace{0.5cm}
\caption{Relative affine root systems of Satake type ${\rm AIII}_{2r}^{(\tau)}$}
	\label{tab1}
\end{table}
%and thus $\tUi$ admits an action of the ibraid group $\text{Br}(W_\tau)$ of type $A_{2r}^{(2)}$; cf. \cite{LWZ23, WZ23, Z23}. 

There are two types of affine rank one subalgebras (i.e., type (\texttt{A}) and (\texttt{B})) appearing in the construction of Drinfeld presentation of $\tUi$ of type AIII$_{2r}^{(\tau)}$. The construction of affine rank one subalgebras of type (\texttt{A}) is standard  (cf. \cite{Be94}); see Proposition~ \ref{prop:rank1isoQG}. On the other hand, the new construction of affine rank one subalgebras of type (\texttt{B}) is different: as $B_i$ for $i\in \I$ have various root lengths, the construction requires a nonstandard generator $\TT_{\bth_r}^{-1} (B_0)$ given in \eqref{eq:Theta_iB0}; see Proposition~\ref{prop:rank1iso-sl3}. 

Under the identification of affine rank one subalgebras with algebras in (\texttt{A}) and (\texttt{B}), the ibraid group actions are shown to match with each other; see Proposition \ref{prop:T1Ti} (generalizing \cite{Be94}). 
With the construction of the affine rank one subalgebras in place, the constructions of $v$-root vectors of $\tUi$ are transported from those for $\tU(\widehat\sll_2)$ \cite{Da93, LWZ24} and those for $\tUiii$ \cite{LWZ23}. 

The ibraid group symmetries feature notably in Drinfeld presentation constructions, in different ways for each family of affine iquanum groups. We work out explicit reduced expressions for the fundamental iweights $\varpi_i$ viewed as elements in the ibraid group $\text{Br}(W_\tau)$; this is consistent with the recursive formulas for counterparts of $\varpi_i$, for $i \neq r$, (obtained by Lusztig \cite{Lus83}) in the setting of affine Hecke algebra of type $\widetilde{C}_r$. The new Drinfeld-type generators, which are distinguished $v$-root vectors of $\tUi$, are defined via ibraid group action. Due to 3 distinct finite rank one subalgebras in $\tUi$, there are many distinct rank two relations to verify in order to show that the algebra with new Drinfeld generators and relations is indeed isomorphic to $\tUi$; see Theorem \ref{thm:Dr}. All these require serious computations involving ibraid group symmetries $\TT_i$ and $\TT_{\varpi_i}$, which occupy Section \ref{sec:relation}. 
\vspace{2mm}

Drinfeld presentation for affine iquantum groups $\tUi$ of type AIII$_{2r+1}^{(\tau)}$ was essential in the geometric realization of $\tUi$ in terms of equivariant K-theory of Steinberg varieties; see Su-Wang \cite{SW24}. The Drinfeld presentation of affine iquantum groups of type AIII$_{2r}^{(\tau)}$ obtained in this paper is expected to have an analogous geometric realization; this will be the subject of a forthcoming paper by L.~Luo, C.~Su and Z.~Xu. Quasi-split affine iquantum groups in Drinfeld presentations naturally lead to the notion of shifted affine iquantum groups, which is expected to be related to geometry of affine Grassmannian islices and iCoulomb branches. In particular, the Drinfeld presentation of type AIII$_{2r}^{(\tau)}$ developed in this paper is expected to be directly applicable to the K-theoretic version of the Coulomb branches associated to Satake diagrams of finite type AIII$_{2r}$ in \cite{SSX25}, and it may have further application to 3D mirror symmetry. 

Drinfeld presentations of affine iquantum groups are expected to lead to $q$-characters just as for the usual affine quantum groups. Some progress has recently been made; cf. \cite{LP25}. 

Drinfeld presentations of twisted affine quantum groups were obtained by Damiani \cite{Da12, Da15}. It will be very interesting yet highly nontrivial to construct Drinfeld presentations for split twisted affine iquantum groups. 
\vspace{2mm}

The paper is organized as follows. 
In Section \ref{sec:iQG}, we review the basics of affine quantum groups and their Drinfeld presentations. We also review affine iquantum groups $\tUi$ of type AIII$_{2r}^{(\tau)}$ and study ibraid group symmetries for $\tUi$.

In Section~\ref{sec:subalgebras}, we construct the affine rank one subalgebras of $\tUi$ and show that they are isomorphic to quantum affine $\mathfrak{sl}_2$ or the affine iquantum group $\tUi(\widehat{\mathfrak{sl}}_3,\tau)$, with compatible braid and ibraid actions. 
	
In Section \ref{sec:Dpresentation}, we construct $v$-root vectors via ibraid group action, which serve as generators for the desired Drinfeld presentaton of $\tUi$. Then we formulate the Drinfeld presentation for $\tUi$, one in the usual commutator form (see Theorem \ref{thm:Dr}) and another in generating function form (see Theorem \ref{thm:DrGF}). The proof that this is indeed a presentation for $\tUi$ is partially given in Section~ \ref{sec:Dpresentation} and completed in Section \ref{sec:relation}, where all the relations for the Drinfeld generators in $\tUi$ are verified. 

	\vspace{2mm}
\noindent {\bf Competing Interests and Funding.}	

The authors have no competing interests to declare that are relevant to the content of this article.
The authors declare that the data supporting the findings of this study are available within the paper.

ML is partially supported by the National Natural Science Foundation of China (No. 12171333, 12261131498). WW is partially supported by the NSF grant DMS-2401351, and he thanks the National University of Singapore (Department of Mathematics and IMS) for providing an excellent research environment and support at the final stage of this project. WZ is partially supported by the New Cornerstone Science Foundation through the New Cornerstone Investigator Program awarded to Xuhua He.

	%%%%%%%%%%
	%%%%%%%%%%
\section{Affne iqauntum groups and ibraid group symmetries}
	\label{sec:iQG}
	
In this section, we review and set up notations for affine iquantum groups and ibraid group symmetries. We review the Drinfeld's presentation for affine quantum groups for type A and for affine rank one iquantum group of type AIII$_{2}^{(\tau)}$. We develop explicitly some properties of the ibraid group symmetries on $\tUi$ of type AIII$_{2r}^{(\tau)}$.

%%%%%%%%%%%%	
\subsection{Affine Weyl and braid groups}

Let $\II=\{1,2,\dots,2r\}$, and $\I=\{0,1,\dots,2r\}$, where $r\geq1$. Let $(c_{ij})_{i,j\in \II}$ be the Cartan matrix of the simple Lie algebra $\fg$ of type $A_{2r}$. Let $\cR_0$ be the set of roots for $\fg$, and fix a set $\cR^+_0$ of positive roots  with simple roots $\alpha_i$ $(i\in \II)$. Denote by $\theta=\sum_{i\in\II}\alpha_i$ the highest root of $\fg$.
	
Let $\widehat{\fg}$ be the (untwisted) affine Lie algebra with affine Cartan matrix denoted by $(c_{ij})_{i,j\in\I}$. Let $\alpha_i$ $(i\in \I)$ be the simple roots of $\widehat{\fg}$, and $\alpha_0=\de -\theta$, where $\de$ denotes the basic imaginary root. The root system $\cR$  for $\widehat{\fg}$ and its positive system $\cR^+$ are given by
\begin{align}
\begin{split}
	\cR &=\{\pm (\beta + k \delta) \mid \beta \in \cR_0^+, k  \in \Z\}  \cup \{m \delta \mid m \in \Z\backslash \{0\} \},
	\\
	\cR^+ &= \{k \delta +\beta \mid \beta \in \cR_0^+, k  \ge 0\}
	\cup  \{k \delta -\beta \mid \beta \in \cR_0^+, k > 0\}
	\cup \{m \delta \mid m \ge 1\}.
\end{split}
		\label{eq:roots}
\end{align}
For $\gamma =\sum_{i\in \I} n_i \alpha_i \in \N \I$, the height $\text{ht} (\gamma)$ is defined as $\text{ht} (\gamma) =\sum_{i\in \I} n_i$.
	
Let $P$ and $Q$ denote the weight and root lattices of the simple Lie algebra $\fg$, respectively. Let $\omega_i \in P$ $(i\in \II)$ be the fundamental weights of $\fg$. Note $\alpha_i =\sum_{j\in \II} c_{ij}\omega_j$. We define a bilinear pairing $\langle \cdot, \cdot \rangle : P\times Q \rightarrow \Z$ such that $\langle \omega_i, \alpha_j \rangle =\delta_{i,j}$, for $i,j \in \II$, and thus $\langle \alpha_i, \alpha_j \rangle = c_{ij}$.

The Weyl group $W_0$ of $\fg$ is generated by the simple reflections $s_i$, for $i \in \II$. They act on $P$ by $s_i(x)=x-\langle x, \alpha_i \rangle\alpha_i$ for $x\in P$. The extended affine Weyl group $\widetilde{W}:=W_0 \ltimes P$ contains the affine Weyl group $W:=W_0 \ltimes Q =\langle s_i \mid i \in \I \rangle$ as a subgroup; we denote
\[
t_\omega =(1, \omega) \in \widetilde W, \quad \text{ for } \omega \in P,
\]
so that $t_{\omega}t_{\omega'} =t_{\omega+\omega'}$. In particular, for $\omega\in P$, $j\in\II$, $t_\omega(\alpha_j)=\alpha_j-\langle \omega,\alpha_j\rangle \de$. We identify $P/Q$ with a finite group $\Omega$ of Dynkin diagram automorphisms, and so $\widetilde{W} =\Omega . W$. There is a length function $\ell(\cdot)$ on $\widetilde{W}$ such that $\ell(s_i)=1$, for $i\in \I$, and each element in $\Omega$ has length 0.
	
For $i\in \II$, we have
\begin{equation}  \label{eq:tomega}
	\ell(t_{\omega_i}) =\ell(\omega_i')+1, \qquad
	\text{ where } \quad \omega_i':= t_{\omega_i} s_i \in W.
\end{equation}
	
\subsection{Quantum groups and iquantum groups}

(This subsection is valid for quantum groups and iquantum groups of Kac-Moody type, though we do not need such a generality.)

Let $\tU =\tU(\widehat{\fg})$ denote the Drinfeld-Jimbo affine quantum group, a $\Q(v)$-algebra generated by $E_i, F_i, \tK_i,\tK_i'$, for $i\in \I$ subject to the following relations:
\begin{align}
	[E_i,F_j]= \delta_{ij} \frac{\tK_i-\tK_i'}{v-v^{-1}},  &\qquad [\tK_i,\tK_j]=[\tK_i,\tK_j']  =[\tK_i',\tK_j']=0,
		\label{eq:KK}
		\\
	\tK_i E_j=v^{c_{ij}} E_j \tK_i, & \qquad \tK_i F_j=v^{-c_{ij}} F_j \tK_i,
		\label{eq:EK}
		\\
	\tK_i' E_j=v^{-c_{ij}} E_j \tK_i', & \qquad \tK_i' F_j=v^{c_{ij}} F_j \tK_i',
		\label{eq:K2}
\end{align}
and the quantum Serre relations (which we skip). Here $\tK_i\tK_i'$ are central in $\tU$, for all $i\in\I$. A central reduction from $\tU$ leads to the Drinfeld-Jimbo quantum group for $\widehat{\fg}$.

For the affine Cartan matrix $C=(c_{ij})_{\I\times\I}$, let $\Aut(C)$ be the group of all permutations $\tau$ of the set $\I$ such that $c_{ij}=c_{\tau i,\tau j}$. 	Let $\tau$ be an involution in $\Aut(C)$, i.e., $\tau^2={\rm Id}$. Following  \cite{LW22}, we define the (universal) iquantum group $\widetilde{\U}^\imath$ to be the $\Q(v)$-subalgebra of $\tU$ generated by
\begin{equation}
		\label{eq:Bi}
B_i= F_i +  E_{\tau i} \tK_i',
\qquad \tk_i = \tK_i \tK_{\tau i}', \quad \forall i \in \I.
\end{equation}
Let $\tU^{\imath 0}$ be the $\Q(v)$-subalgebra of $\tUi$ generated by $\tk_i$, for $i\in \I$. The elements $\tk_i$ (for $i= \tau i$) and $\tk_i \tk_{\tau i}$  (for $i\neq \tau i$) are central in $\tUi$. Moreover, $\U^\imath$ is a right coideal subalgebra of $\U$ in the sense that $\Delta: \tUi \rightarrow \tUi\otimes \tU$; and $(\tU,\tUi)$ forms a quantum symmetric pair. The iquantum groups \`a la Letzter (cf. \cite{Let99, Ko14}) are obtained from $\tUi$ by central reductions, and will not be used in this paper. 

%Let $\bvs=(\vs_i)\in  (\Q(v)^\times)^{\I}$ be such that $\vs_i=\vs_{\tau i}$ for all $i$. Let $\Ui:=\Ui_{\bvs}$ be the $\Q(v)$-subalgebra of $\U$ generated by
%\[ B_i= F_i+\vs_i E_{\tau i}K_i^{-1},	\quad
%k_j= K_jK_{\tau j}^{-1},
%\qquad  \forall i \in \I, j \in \I\backslash\I_\tau. \]

%\begin{proposition} [{\cite[Proposition 6.2]{LW22}}]
%The algebra $\Ui$ is isomorphic to the quotient of $\tUi$ by the ideal generated by
%\begin{align*}
%	\tk_i - \vs_i \; (\text{for } i =\tau i),	\qquad  \tk_i \tk_{\tau i} - \vs_i \vs_{\tau i}  \;(\text{for } i \neq \tau i).
%\end{align*}
%The isomorphism is given by sending $B_i \mapsto B_i, k_j \mapsto \vs_{\tau j}^{-1} \tk_j, k_j^{-1} \mapsto \vs_{j}^{-1} \tk_{\tau j}, \forall i\in \I, j\in \I\backslash\I_\tau$. Moreover, the algebra $\widetilde{\U}^\imath$ is a right coideal subalgebra of $\widetilde{\U}$.
%\end{proposition}

We shall refer to $\tUi$ as {\em (quasi-split) $\mathrm{i}${}quantum groups}; they are called {\em split} if $\tau =\mathrm{Id}$.
For any $i\in\I$, we set
\begin{align*}
	\K_i:=-v^{2}\tk_i, \text{ if }\tau i=i;
		\qquad
	\K_j=\tk_j, \text{ otherwise.}
\end{align*}
For any $\alpha=\sum_{i\in\I} a_i\alpha_i\in\Z\I$, we set
\begin{align*}
       % \label{eq:bbKi}
		\K_\alpha:=\prod_{i\in\I}\K_i^{a_i}.
\end{align*}

\subsection{Relative affine Weyl/braid groups}
	\label{sub:Weyl}
	
The relative root systems and relative Weyl groups (of finite type) are well known; we refer to the exposition in \cite[\S2.3]{DK19} and the references therein. In this subsection we shall adapt this to set up an affine version of relative root systems and relative Weyl groups which are needed in this paper; see \cite{Lus03, LWZ24}.

Given a quasi-split Satake diagram $(\I,\tau)$ of Kac-Moody type, we fix
\begin{align}   \label{eq:ci}
	\I_\tau &= \{ \text{the chosen representatives of $\tau$-orbits in $\I$} \}.
\end{align} 
We denote by $\bs_{i}$ the following elements of order 2 in the Weyl group $W =\langle s_i\mid i\in\I \rangle$, i.e.,
	\begin{align}
		\label{def:simple reflection}
		\bs_i= \left\{
		\begin{array}{ll}
			s_{i}, & \text{ if } c_{i,\tau i}= 2;
			\\
			s_is_{\tau i}, & \text{ if }  c_{i,\tau i}=0;
			\\
			s_is_{\tau i}s_i, &\text{ if }  c_{i,\tau i}=-1.
		\end{array}
		\right.
	\end{align}
Note that $\bs_i=\bs_{\tau i}$ for any $i\in\I$ and hence we can parametrize the $\bs_i$ by $i\in\I_\tau$. Consider the following subgroup $W_\tau$ of the Weyl group $W$:
	\begin{align}
		\label{eq:Wtau}
		W_{\tau} =\{w\in W\mid \tau w =w \tau\}.
	\end{align}
The group $W_{\tau}$ is a Coxeter group with $\bs_i$ ($i\in \I_\tau$) as its generating set (cf. \cite{Lus03}; compare with \cite[\S2.3]{DK19}).

Now we specialize to the affine type with $\I = \II \cup \{0\}$. We shall always assume that the involution $\tau$ fixes the affine node $0$. Recall from \eqref{eq:roots} that $\cR$ is an affine root system. Define the element $\balpha\in \Q\cR$ by
	\begin{equation}
		\balpha:=\frac{\alpha+\tau \alpha}{2}, \quad(\alpha\in\cR).
	\end{equation}
Note that $\balpha_i=\balpha_{\tau i}$ for $i\in\I$ and $\balpha_0=\alpha_0$. 
	Let $\ov{\cR}:=\{\balpha\mid\alpha\in \cR\}$ be the relative affine root system associated to the quasi-split affine symmetric pair $(\widehat\fg, \widehat\fg^{\omega\tau})$. Then $\ov{\cR}$ admits a simple system $\{\balpha_i\mid i\in\I_\tau\}$ and the corresponding positive system $\ov{\cR}^+=\{\balpha\mid\alpha\in \cR^+\}$. Let $(\ov{c}_{ij})_{i,j\in \I_\tau}$ be the Cartan matrix of the relative root system, where $\ov{c}_{ij} =\frac{2(\balpha_i,\balpha_j)}{(\balpha_i,\balpha_i)}$. 

Then $W_\tau$ is the Weyl group associated to the root system $\ov{\cR}$. 
We shall refer to $W_{\tau}$ as the {\em relative affine Weyl group} associated with the affine symmetric pair $(\widehat\fg, \widehat\fg^{\omega\tau})$. Note that $W_{\tau}^0 =\{w\in W^0\mid \tau w =w \tau\}$ is the relative finite Weyl group associated with the symmetric pair $(\fg, \fg^{\omega\tau})$ with a generating set $\{\bs_i \mid i\in \I_{0,\tau}\}$. Let $\ell^\circ(\cdot)$ be the \emph{length function} of $W_\tau$. 

Recalling $\I_\tau$ from \eqref{eq:ci}, we set
	\begin{align*}
		\I_{0,\tau} = \I_{\tau} \cap \I_0.
	\end{align*}	
Define $\bome_i$, for $i\in \I_{0,\tau}$, to be the following elements in the extended affine Weyl group $\widetilde{W}$:
	\begin{align}\label{def:bome}
		\bome_{i}= 
			t_{\omega_i}t_{\omega_{\tau i}}.
	\end{align}
Then $\bome_i\in W_\tau$.
	\begin{lemma}
		For $i\in\II$, any reduced presentation of $\bome_{i}$ ends with $\bs_i$.
	\end{lemma}
	
	\begin{proof}
		If $\ell(\bome_i\bs_j)<\ell(\bome_i)$ then $\bome_i(\balpha_j)<0$, which happens exactly when $i=j$.
	\end{proof}

\subsection{iBraid group symmetries}

The braid group associated to the relative affine Weyl group for $(\widehat\fg, \widehat\fg^{\omega\tau})$, where $\omega$ denotes the Chevalley involution, is of the form
	\begin{equation}
		\label{eq:braidCox}
		\brW =\langle t_i \mid i\in \I_\tau \rangle, 
	\end{equation}
where $t_i$ satisfy the same braid relations as for $\bs_i$ in $W_{\tau}$. (The following ibraid group symmetries are actually valid for quasi-split iquantum groups of Kac-Moody type.) 
	
\begin{theorem} [\cite{Z23}; \text{also cf. \cite{KP11, LW21a, LW22b, WZ23, LWZ23}}]
		\label{thm:Ti}
	(1) For $i\in \I$ such that $\tau i=i$, there exists an automorphism $\TT_i$ of the $\Q(v)$-algebra $\tUi$ such that
		$\TT_i(\K_\mu) =\K_{\bs_i\mu}$, and
		\[
		\TT_i(B_j)= \begin{cases}
			\K_i^{-1} B_i,  &\text{ if }j=i,\\
			B_j,  &\text{ if } c_{ij}=0, \\
			B_jB_i-vB_iB_j,  & \text{ if }c_{ij}=-1, \\
			{[}2]^{-1} \big(B_jB_i^{2} -v[2] B_i B_jB_i +v^2 B_i^{2} B_j \big) + B_j\K_i,  & \text{ if }c_{ij}=-2,
		\end{cases}
		\]
		for $\mu\in \Z\I$ and $j\in \I$.
		
	(2) For $i\in \I$ such that $c_{i,\tau i}=0$, there exists an automorphism $\TT_i$ of the $\Q(v)$-algebra $\tUi$ such that, for any $j\in\I$,
		$\TT_i(\K_j) =(-v)^{-c_{ij}-c_{\tau i,j}}\K_{j}\K_i^{-c_{ij}} \K_{\tau i}^{-c_{\tau i,j}}$, and
		\[
		\TT_i(B_j)= \begin{cases}
			B_jB_i -vB_iB_j,  & \text{ if }c_{ij}=-1 \text{ and } c_{\tau i,j}=0,
			\\
			B_jB_{\tau i}-v B_{\tau i}B_j,  & \text{ if } c_{ij}=0 \text{ and }c_{\tau i,j}=-1 ,
			\\
			[[B_j,B_i]_v,B_{\tau i}]_v - v B_j\bK_{i},  & \text{ if } c_{ij}=-1 \text{ and }c_{\tau i,j}=-1 ,
			\\
			-\bK_{i}^{-1} B_{\tau i},  & \text{ if }j=i,
			\\
			-\bK_{\tau i}^{-1} B_i,  &\text{ if }j=\tau i,
			\\
			B_j, & \text{ otherwise.}
		\end{cases}
		\]
%		for $j\in \I$.
		
	(3) For $i\in \I$ such that $c_{i,\tau i}=-1$, there exists an automorphism $\TT_i$ of the $\Q(v)$-algebra $\tUi$ such that, for any $j\in\I$,
		\[
		\TT_i(\K_j)=v^{-(c_{ij}+c_{\tau i,j})}\K_j(\K_i\K_{\tau i})^{-c_{ij}-c_{\tau i,j}},
		\]
		and
		\[
		\TT_i(B_j)= \begin{cases}
			\big[ [B_j,B_i]_v,B_{\tau i} \big]_v-\bK_i B_{j},  & \text{ if }c_{ij}=-1 \text{ and } c_{\tau i,j}=0,\\
			\big[ [B_j,B_{\tau i}]_v,B_{i} \big]_v- \bK_{\tau i} B_{j},  & \text{ if } c_{ij}=0 \text{ and }c_{\tau i,j}=-1 ,\\
			v\Big[\big[ [B_j,B_i]_v,B_{\tau i}\big] ,[B_{\tau i},B_i]_v\Big]\\
			\qquad-\big[B_j, [B_{\tau i},B_i]_{v^3} \big]\bK_{i}+ vB_j\bK_i\bK_{\tau i},  & \text{ if } c_{ij}=-1 \text{ and }c_{\tau i,j}=-1 ,\\
			-v^{-2}B_{i}\bK_{\tau i}^{-1},  & \text{ if }j=i,\\
			-v^{-2} B_{\tau i}\bK_{i}^{-1},  &\text{ if }j=\tau i,\\
			B_j, & \text{ otherwise.}
		\end{cases}
		\]
		
Moreover, there exists a homomorphism $\brW \rightarrow \Aut( \tUi)$, $t_i\mapsto \TT_i$, for all $i\in \I_\tau$.
	\end{theorem}

\begin{remark}
Let $\Phi$ be the rescaling automorphism on $\tUi$ such that
	\begin{align*}
		B_i\mapsto B_i,\quad B_{\tau i}\mapsto -v^{-1} B_{\tau i}, \quad\K_i\mapsto -v^{-1}\K_i,\quad\K_{\tau i} \mapsto -v^{-1}\K_{\tau i}
	\end{align*}
for $i\neq \tau i, i\in \I_\tau$ and $\Phi$ fixes $B_j,\K_j$ for $j=\tau j$. Then the symmetries $\mathrm{T}_i$ in \cite[\S 5]{LW21a}, for $c_{i,\tau i}=0,2$, is given by $\mathrm{T}_i=\Phi \TT_i \Phi^{-1}$. 
\end{remark}

The quantum group version of the following statement is well known. 	
\begin{proposition}[\text{cf. \cite[Theorem 7.13]{WZ23}; also see \cite[Lemma 2.9]{LWZ24}}]
		\label{prop:fixB}
	We have $\TT_w (B_i) = B_{w i}$, for $i\in \I$ and $w \in W_\tau$ such that $wi \in \I$.
\end{proposition}
    
There exists a $\Q(v)$-algebra anti-involution $\sigma_\imath: \tUi\rightarrow \tUi$ such that
	\begin{align}
		\label{eq:sigma}
		\sigma_\imath(B_i)=B_{i}, \quad \sigma_\imath(\K_i)= \K_{\tau i},
		\quad \forall i\in \I.
	\end{align}
		
By \cite[Theorem 6.7]{WZ23}, we have
	\begin{align}
		\TT_i^{-1}=\sigma_\imath \TT_i\sigma_{\imath}, \quad\forall i\in\I.
	\end{align}

\subsection{Affine iquantum groups of quasi-split type AIII$_{2r}^{(\tau)}$}
    
In the remainder of this paper, we only consider the affine iquantum group with Satake diagram \eqref{pic:AIII2r}.	
The affine iquantum group of type ${\rm AIII}_{2r}^{(\tau)}$ admits the following presentation \cite{CLW21}: it is isomorphic to the $\Q(v)$-algebra $\tUi =\tUi(\widehat{\fg})$ with generators $B_i$, $\K_i^{\pm 1}$ $(i\in \I)$, subject to the following relations, for $i, j\in \I$:
	\begin{align}
		\K_i\K_i^{-1} =\K_i^{-1}\K_i=1,  \quad \K_i\K_\ell &=\K_\ell \K_i, \quad
		\K_\ell B_i=v^{c_{\tau \ell,i} -c_{\ell i}} B_i \K_\ell,
		\\
		B_iB_j -B_j B_i&=0, \qquad\qquad\qquad \text{ if } c_{ij}=0 \text{ and }\tau i\neq j,
		\label{eq:S1} \\
		\sum_{s=0}^{1-c_{ij}} (-1)^s \qbinom{1-c_{ij}}{s} & B_i^{s}B_jB_i^{1-c_{ij}-s} =0, \quad \text{ if } j \neq \tau i\neq i,
		\label{eq:S6}
		\\
		B_{\tau i}B_i -B_i B_{\tau i}& =   \frac{\K_i -\K_{\tau i}}{v-v^{-1}},
		\quad \text{ if }  c_{i,\tau i}=0,
		\label{relation5}	\\
		B_i^2 B_j -[2] B_i B_j B_i +B_j B_i^2 &= - v^{-1}  B_j \K_i,  \qquad \text{ if }c_{ij}=-1 \text{ and }c_{i,\tau i}=2,
		\label{eq:S2} \\
		B_{i}^2 B_{\tau i} -[2]B_i B_{\tau i}B_i +B_{\tau i}B_i^2&= -[2](v \bK_iB_i+v B_i\bK_{\tau i}),
		\qquad  \text{ if } c_{i,\tau i}=-1.
		\label{eq:S4} 
	\end{align}
	
The involution $\tau$ induces an involution of $\tUi$, which is denoted by $\widehat{\tau}$: $B_i\leftrightarrow B_{\tau i}$, $\K_i\leftrightarrow \K_{\tau i}$.
	
%For $\mu = \mu' +\mu''  \in \Z \I := \oplus_{i\in \I} \Z \alpha_i$,  define $\K_\mu$ such that
%	\begin{align}
%		\K_{\alpha_i} =\K_i, \quad
%		\K_{-\alpha_i} =\K_i^{-1}, \quad
%		\K_{\mu} =\K_{\mu'} \K_{\mu''},
%		\quad  \K_\delta =\K_0 \K_\theta.
%	\end{align}
The algebra $\tUi$ is endowed with a filtered algebra structure
	\begin{align}  \label{eq:filt1}
		\widetilde{\U}^{\imath,0} \subset \widetilde{\U}^{\imath,1} \subset \cdots \subset \widetilde{\U}^{\imath,m} \subset \cdots
	\end{align}
by setting 
	\begin{align}  \label{eq:filt}
		\widetilde{\U}^{\imath,m} =\Q(v)\text{-span} \{ B_{i_1} B_{i_2} \ldots B_{i_s} \K_\mu \mid \mu \in \Z\I, i_1, \ldots, i_s \in \I, s \le m \}.
	\end{align}
Note that
\begin{align}  \label{eq:UiCartan}
	\widetilde{\U}^{\imath,0} =\bigoplus_{\mu \in \Z\I} \Q(v) \K_\mu
\end{align}
is the $\Q(v)$-subalgebra generated by $\K_i$ for $i\in \I$. Then, according to a basic result of Letzter and Kolb on quantum symmetric pairs adapted to our setting of $\tUi$ (cf. \cite{Let02, Ko14}), the associated graded $\gr \tUi$ with respect to \eqref{eq:filt1}--\eqref{eq:filt} can be identified with
\begin{align}   \label{eq:filter}
	\gr \tUi \cong \U^- \otimes \Q(v)[\K_i^\pm | i\in \I],
	\qquad \overline{B_i}\mapsto F_i,  \quad
	\overline{\K}_i \mapsto \K_i \; (i\in \I).
\end{align}

\subsection{Properties of ibraid group operators}

The relative affine root system and relative affine Weyl group for $\tUi$ of type ${\rm AIII}_{2r}^{(\tau)}$ can be read off from Table \ref{tab1}; they are of type $A_{2r}^{(2)}$. 

\begin{lemma}[{\cite[Lemma 4.4]{Lus83}}] 
		\label{lem:Tomeij}
	We have
	\begin{itemize}
		\item[(1)] $\TT_{\bome_i}\TT_{\bome_j}=\TT_{\bome_j}\TT_{\bome_i}$, for any $i,j\in\I_{0}$;
		\item[(2)] $\TT_i^{-1} \TT_{\bome_i} \TT_i^{-1}=\TT_{\bome_i}^{-1} \prod_{k\neq i,k\in \I_{0,\tau}} \TT_{\bome_k}^{-\ov{c}_{ik}}$, for $i\neq r,r+1$ in $\I_0$. 
	\end{itemize}
\end{lemma}
	
\begin{lemma}[{cf. \cite[Lemma 2.11]{LWZ24}}]
		\label{lem:T-T}
	We have
		$\TT_{\bome_i'}(B_i) =\TT_{\bome_i'}^{-1}(B_i),$ for $i\in \II$ and $i\neq r,r+1$.
\end{lemma}
	
\begin{proof}
It is the same as the proof for \cite[Lemma 2.11]{LWZ24} using Lemma \ref{lem:Tomeij}(2) now. 
\end{proof}

For type ${\rm A}_{2r}^{(2)}$, the extended relative affine Weyl group coincides with relative affine Weyl group $W_\tau$.	
The following recursive formula of type ${\rm A}_{2r}^{(2)}$ (for $1\le k<r$) coincides with its counterpart in type $\widetilde{C}_r$ in \cite[\S4.5]{Lus83}, but they differ for $k=r$. By convention, we set $\TT_{\bome_0}=1$.

\begin{lemma}
For $1\le k \le r$, we have
    \begin{align}\label{TT:induction}
        \TT_{\bome_k} =\TT_{\bome_{k-1}} \TT_{k-1}^{-1}\cdots\TT_2^{-1} \TT_1^{-1} \TT_{\bome_1} \TT_1^{-1}\TT_2^{-1}\cdots \TT_{k-1}^{-1}.
    \end{align}
\end{lemma}

\begin{proof}
The formula is trivial for $k=1$. Assume that $2\le k\le r.$ By Lemma \ref{lem:Tomeij}, we have
\[
\TT_{k-1}^{-1}\TT_{\bome_{k-1}}\TT_{k-1}^{-1}=\TT_{\bome_{k-2}} \TT_{\bome_{k-1}}^{-1} \TT_{\bome_k}.
\]
Hence we have
\begin{align}
\label{eq:induct-weyl}
\TT_{\bome_{k-1}}^{-1} \TT_{\bome_k} = \TT_{\bome_{k-2}}^{-1} \TT_{k-1}^{-1}\TT_{\bome_{k-1}}\TT_{k-1}^{-1}
= \TT_{k-1}^{-1} \TT_{\bome_{k-2}}^{-1} \TT_{\bome_{k-1}}\TT_{k-1}^{-1}.
\end{align}
Now the lemma follows by a simple induction. 
\end{proof}

%The following formula is specific for quasi-split affine type AIII$_{2r}^{(\tau)}$.
	\begin{proposition}\label{prop:TT_bome}
		For $1\le k \le r$, we have
		\begin{align}
			\label{eq:Tome}
			\TT_{\bome_k}&=(\TT_0\TT_1\cdots \TT_{r-1}\TT_{r}\TT_{r-1}\cdots \TT_k)^k.
		\end{align}
	\end{proposition}
	
	\begin{proof}
		Let us show \eqref{eq:Tome} for $k=1$. Recall from \eqref{def:bome} that that $\bome_{1}=t_{\omega_1}t_{\omega_{2r}}$. In the Weyl group of affine type $A_{2r}$, it is known that $t_{\omega_1}=\rho s_{2r}\cdots s_{2}s_1$, $t_{\omega_{2r}}=\rho^{2r}s_1s_2\cdots s_{2r}$ are reduced expressions; cf. \cite[\S 4.5]{Lus83}.
		Here $\rho:\I\rightarrow \I$ is the diagram automorphism which sends $i$ to $i+1$ modulo $2r+1$.  Then 
		\begin{align*}
			\bome_{1}=t_{\omega_{2r}}t_{\omega_1}
			&=\rho^{2r} s_1s_2\cdot s_{2r}\rho s_{2r}\cdots s_{2}s_1
			\\
			&=\bs_0\bs_1\cdots \bs_{r-1}\bs_r\bs_{r-1}\cdots\bs_1.
		\end{align*}
		Recall that $\ell,\ell^\circ$ denote length functions in $W,W_\tau$ respectively. It is clear that 
		\begin{align*}
			\ell(\bome_1)=\ell(t_{\omega_1})+\ell(t_{\omega_{2r}})=4r=\sum\ell(\bs_i),
		\end{align*}
		where the summation on RHS runs over the expression $\bs_0\bs_1\cdots \bs_{r-1}\bs_r\bs_{r-1}\cdots\bs_1$. It follows that  $\ell^\circ(\bome_1)=\sum\ell^\circ(\bs_i)$, which implies that $\bome_1=\bs_0\bs_1\cdots \bs_{r-1}\bs_r\bs_{r-1}\cdots\bs_1$ is reduced.
		Thus, \eqref{eq:Tome} for $k=1$ follows.  

Using \eqref{eq:Tome} for $k=1$, we rewrite \eqref{TT:induction} as
        \begin{align}\label{TT:ind}
        \TT_{\bome_k}=\TT_{\bome_{k-1}}\TT_{k-1}^{-1}\cdots\TT_2^{-1} \TT_1^{-1} (\TT_0\TT_1\cdots \TT_{r-1}\TT_{r}\TT_{r-1}\cdots \TT_k).
    \end{align}

		Now we prove \eqref{eq:Tome} for $k>1$.         
        By a direct computation, we have for $1\leq j<k-1$
        \begin{align}\label{TT:ind2}
        (\TT_0\TT_1\cdots \TT_{r-1}\TT_{r}\TT_{r-1}\cdots \TT_{k})\TT_j^{-1}
        =
        \TT_{j+1}^{-1}(\TT_0\TT_1\cdots \TT_{r-1}\TT_{r}\TT_{r-1}\cdots \TT_{k}).
        \end{align}
        The desired formula \eqref{eq:Tome} follows by \eqref{TT:ind}--\eqref{TT:ind2} by induction on $k$. 
	\end{proof}

  We record the following special cases which will be used later; the first two formulas in the context of affine Hecke algebra of type $A_{2r}^{(2)}$ can also be found in \cite[Corollary 4.2.4]{Da00}.
    \begin{corollary} \label{cor:Da}
    We have
    \begin{align}
    \label{eq:Tomega1}
	\TT_{\bome_1}&=\TT_0\TT_1\cdots \TT_{r-1}\TT_{r}\TT_{r-1}\cdots \TT_1,\\
    \label{eq:Tomegar-1}
	\TT_{\bome_{r-1}}&=(\TT_0\TT_1\cdots \TT_r)^{r-1}\TT_1\TT_2\cdots \TT_{r-1},
    \\
    \TT_{\bome_{r}} &=(\TT_0\TT_1\TT_2\cdots\TT_r)^r.
    \label{eq:Tomega3}
    \end{align}
    \end{corollary}

    \begin{proof}
     The formula \eqref{eq:Tomega1} and \eqref{eq:Tomega3} directly follow from Proposition~\ref{prop:TT_bome}. 
     
     We derive \eqref{eq:Tomegar-1} from Proposition~\ref{prop:TT_bome} and its proof. The case $r=2$ is trivial and hence we assume $r>2$. Write $\TT_{\bs_0\bs_1\cdots\bs_r}:=\TT_0\TT_1\cdots \TT_{r-1}\TT_{r}$. Setting $k=r$ in \eqref{TT:ind2}, we have $\TT_{j+1}\TT_{\bs_0\bs_1\cdots\bs_r}=\TT_{\bs_0\bs_1\cdots\bs_r}  \TT_j $ for $1\leq j<r-1$ and then
    \begin{align}\label{TT:ind3}
        \TT_{r-1}(\TT_{\bs_0\bs_1\cdots\bs_r})^l
        =
         (\TT_{\bs_0\bs_1\cdots\bs_r})^l \TT_{r-l-1}.
    \end{align} 
    Hence, by Proposition~\ref{prop:TT_bome} and \eqref{TT:ind3}, we have
    \begin{align*}
    \TT_{\bome_{r-1}}&=(\TT_{\bs_0\bs_1\cdots\bs_r}\TT_{r-1})^{r-1}
    \\
    %&=( \TT_{\bs_0\bs_1\cdots\bs_r}\TT_{r-1})^{r-2}\TT_{\bs_0\bs_1\cdots\bs_r}\TT_{r-1}
    %\\
    &=( \TT_{\bs_0\bs_1\cdots\bs_r}\TT_{r-1})^{r-3}\TT_{\bs_0\bs_1\cdots\bs_r}\TT_{r-1}\TT_{\bs_0\bs_1\cdots\bs_r}\TT_{r-1}
    \\
    &=( \TT_{\bs_0\bs_1\cdots\bs_r}\TT_{r-1})^{r-3}(\TT_{\bs_0\bs_1\cdots\bs_r})^2\TT_{r-2}\TT_{r-1}
    \\
    &=( \TT_{\bs_0\bs_1\cdots\bs_r}\TT_{r-1})^{r-4}\TT_{\bs_0\bs_1\cdots\bs_r}\TT_{r-1} (\TT_{\bs_0\bs_1\cdots\bs_r})^2\TT_{r-2}\TT_{r-1}
    \\
    &=( \TT_{\bs_0\bs_1\cdots\bs_r}\TT_{r-1})^{r-4}(\TT_{\bs_0\bs_1\cdots\bs_r})^3\TT_{r-3}\TT_{r-2}\TT_{r-1}
    \\
    &=\cdots
    =(\TT_{\bs_0\bs_1\cdots\bs_r})^{r-1}\TT_1\TT_2\cdots \TT_{r-1},
    \end{align*}
    as desired.
    \end{proof}

    \begin{corollary}
    \label{cor:length}
    Let $1\le k \le r$. We have the following reduced expressions of $\bome_k$ in $W_\tau$: 
    \[
    \bome_k = (\bs_0\bs_1\cdots \bs_{r-1}\bs_{r}\bs_{r-1}\cdots \bs_k)^k;
    \] 
    and in addition, ${\bome_{r-1}}=(\bs_0\bs_1\cdots \bs_r)^{r-1}\bs_1\bs_2\cdots \bs_{r-1}$. In particular, the length of $\bome_k$ in $W_\tau$ is  $k(2r+1-k)$. 
    \end{corollary}

    \begin{proof}
    By \cite[\S 4.5]{Lus83}, $\ell(t_{\omega_i})=i(2r+1-i)$ and then $\ell(\bome_k)=2k(2r+1-k)$. On the other hand, since $\ell(\bs_0)=1,\ell(\bs_r)=3,\ell(\bs_i)=2$ for $1\le i\le r-1$, the length of $(\bs_0\bs_1\cdots \bs_{r-1}\bs_{r}\bs_{r-1}\cdots \bs_k)^k$ in $W$ is also $2k(2r+1-k)$. Hence, this is a reduced expression of $\bome_k$.
    \end{proof}

\subsection{Drinfeld presentation of affine quantum groups}
	\label{subsec:Drsl2}
	
The affine quantum group $\U$ admits a second presentation known as the Drinfeld presentation. Recall $C=(c_{ij})_{i,j\in\I_0}$ is the Cartan matrix of the simple Lie algebra $\g$. Let $^{\text{Dr}}\U$ be the $\Q(v)$-algebra generated by $x_{i k}^{\pm}$, $h_{i l}$, $K_i^{\pm 1}$, $\texttt{C}^{\pm1}$, for $i\in\I_0$, $k\in\Z$, and $l\in\Z\backslash\{0\}$, subject to the following relations: $\texttt{C}^{\pm1}$ are central and
\begin{align*}
	[K_i,K_j] & =  [K_i,h_{j l}] =0, \quad K_i K_i^{-1} =\texttt{C} \texttt{C}^{- 1} =1,
		\\
	[h_{ik},h_{jl}] &= \delta_{k, -l} \frac{[k c_{ij}]}{k} \frac{\texttt{C}^k -\texttt{C}^{-k}}{v -v^{-1}},
		\\
	K_ix_{jk}^{\pm} K_i^{-1} &=v^{\pm c_{ij}} x_{jk}^{\pm},
		\\
	[h_{i k},x_{j l}^{\pm}] &=\pm\frac{[kc_{ij}]}{k} \texttt{C}^{\frac{k\mp |k|}{2}}x_{j,k+l}^{\pm},
		\\
	[x_{i k}^+,x_{j l}^-] &=\delta_{ij} {(\texttt{C}^{-l} K_i\psi_{i,k+l} - \texttt{C}^{-k} K_i^{-1} \varphi_{i,k+l})}, %{(v_i-v_i^{-1})},
		\\
	x_{i,k+1}^{\pm} x_{j,l}^{\pm}-v^{\pm c_{ij}} x_{j,l}^{\pm} x_{i,k+1}^{\pm} &=v^{\pm c_{ij}} x_{i,k}^{\pm} x_{j,l+1}^{\pm}- x_{j,l+1}^{\pm} x_{i,k}^{\pm},
		\\
	\Sym_{k_1,\dots,k_r}\sum_{t=0}^{r} (-1)^t \qbinom{r}{t} & x_{i,k_1}^{\pm}\cdots
	x_{i,k_t}^{\pm} x_{j,l}^{\pm}  x_{i,k_t+1}^{\pm} \cdots x_{i,k_n}^{\pm} =0, \text{ for } r= 1-c_{ij}\; (i\neq j),
\end{align*}
where $\Sym_{k_1,\dots,k_r}$ denotes the symmetrization with respect to the indices $k_1,\dots,k_r$, $\psi_{i,k}$ and $\varphi_{i,k}$ are defined by the following equations:
\begin{align*}
	1+ \sum_{m\geq 1} (v-v^{-1})\psi_{i,m}u^m &=  \exp\Big((v -v^{-1}) \sum_{m\ge 1}  h_{i,m}u^m\Big),
		\\
	1+ \sum_{m\geq1 } (v-v^{-1}) \varphi_{i, -m}u^{-m} &= \exp \Big((v^{-1} -v) \sum_{m\ge 1} h_{i,-m}u^{-m}\Big).
	\end{align*}
(We omit a degree operator $D$ in the version of ${}^{\text{Dr}}\U$ above.) There exists an isomorphism of $\Q(v)$-algebras ${}^{\text{Dr}}\U \cong \U$; cf. \cite{Dr87, Be94, Da15}.

\subsection{Drinfeld presentation of iquantum group of type AIII$_{2}^{(\tau)}$}
	\label{sec:rank1}
		 
Let $\tau$ be the following diagram automorphism given by swapping vertices $1$ and $2$ while fixing $0$: 
	\begin{center}\setlength{\unitlength}{0.9mm}
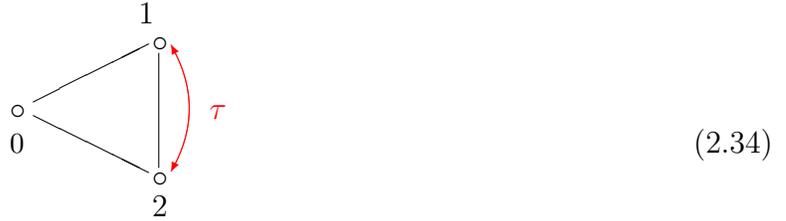
		\begin{equation}
			\label{eq:satakerank1}
			\begin{picture}(40,20)(-20,15)
				\put(0,10){$\circ$}
				\put(0,30){$\circ$}
				\put(0,5.5){$2$}
				\put(-2,34){${1}$}
				%	\put(92,16){\small $r$}
				\put(-17.5,22.5){\line(2,1){17}}
				\put(-17.5,20.5){\line(2,-1){17}}
				\put(1,12.7){\line(0,1){17}}
				
				\put(-21,20){$\circ$}
				\put(-21,15){\small $0$}
				
				\color{red}
                \put(8.5,20){$\tau$}
				\qbezier(3,12.5)(8,21.5)(3,30.5)
				\put(3.25,13){\vector(-1,-2){0.5}}
				\put(3.25,30){\vector(-1,2){0.5}}
			\end{picture}
		\end{equation}     
		\vspace{0.5cm}
	\end{center}	

We recall from \cite{LWZ23} a Drinfeld type presentation for $\tUi(\widehat{\mathfrak{sl}}_3,\tau)$ of quasi-split affine rank one type $A_2^{(\tau)}$ associated with the Satake diagram \eqref{eq:satakerank1}. Let $\Sym_{k_1,k_2}$ denote the symmetrization with respect to indices $k_1, k_2$ in the sense $\Sym_{k_1,k_2} f(k_1, k_2) =f(k_1, k_2) +f(k_2, k_1)$.
	
\begin{definition}
		\label{def:iDR}
Let $\tUiD(\widehat{\mathfrak{sl}}_3,\tau)$ be the $\Q(v)$-algebra generated by the elements $B_{i,l}$, $H_{i,m}$, $\bK_i^{\pm1}$, $C^{\pm1}$, where $i=1,2$, $l\in\Z$ and $m \in \Z_{\ge 1}$, subject to the following relations: for $m, n \ge 1, k, l \in \Z$, and $i, j \in \{1,2\}$,
\begin{align}
	C \text{ is central,} \quad &
	\K_i\K_j=\K_j\K_i, \quad
	\K_i H_{j,m}=H_{j,m}\K_i,\quad
	\bK_iB_{j,l}=v^{c_{\tau i,j}-c_{ij}} B_{j,l} \bK_i,
		\label{qsiA1DR1} \\
	[H_{i,m},H_{j,n}] &=0,\label{qsiA1DR2}
			\\
	[H_{i,m},B_{j,l}] &=\frac{[mc_{ij}]}{m} B_{j,l+m}-\frac{[mc_{\tau i,j}]}{m} B_{j,l-m}C^m,\label{qsiA1DR3}
			\\
		\label{qsiA1DR4}
	[B_{i,k},B_{i,l+1}]_{v^{-2}} & -v^{-2}[B_{i,k+1},B_{i,l}]_{v^{2} }=0,
			\\
		\label{qsiA1DR5}
	[B_{i,k},B_{\tau i,l+1}]_v & -v[B_{i,k+1},B_{\tau i,l}]_{v^{-1}} = -\Theta_{{\tau i},l-k+1}C^k \bK_{i} +v \Theta_{ {\tau i},l-k-1}C^{k+1}\bK_{i}
			\\
	& \qquad\qquad\qquad\qquad  -\Theta_{i,k-l+1}C^l\bK_{{\tau i}} +v \Theta_{i,k-l-1}C^{l+1}\bK_{\tau i},
			\notag  \\
	\bS_{i,\tau i}(k_1,k_2|l)
	=[2]&\Sym_{k_1,k_2}\sum_{p\geq 0}v^{2p}
	\big[\TH_{\tau i,l-k_2-p}\K_i-v\TH_{\tau i,l-k_2-p-2}C\K_i, B_{i,k_1-p} \big]_{v^{-4p-1}}C^{k_2+p}
			\notag
			\\
	+v[2]&\Sym_{k_1,k_2}\sum_{p\geq 0}v^{2p} \big[ B_{i,k_1+p+1},\TH_{i,k_2-l-p+1}\K_{\tau i}-v\TH_{i,k_2-l-p-1}C \K_{\tau i}\big]_{v^{-4p-3}} C^{l-1}.
			\label{qsiA1DR6}
\end{align}
Here $H_{i,m}$ are related to $\Theta_{i,m}$ by the following equation:
\begin{align}
	\label{exp h}
	1+ \sum_{m\geq 1} (v-v^{-1})\Theta_{i,m} u^m  = \exp\Big( (v-v^{-1}) \sum_{m\geq 1} H_{i,m} u^m \Big).
\end{align}
We have also denoted
\begin{align}
	\label{eq:SS}
\bS_{i,\tau i}(k_1,k_2|l) : = \Sym_{k_1,k_2}\Big(B_{i,k_1}B_{i,k_2}B_{\tau i,l}-[2]B_{i,k_1} B_{\tau i,l} B_{i,k_2} + B_{\tau i,l} B_{i,k_1}B_{i,k_2}\Big).
\end{align}
\end{definition}
	
Fix  signs $o(1)$ and $o(2)$ associated to the nodes $\I_0=\{1,2\}$ such that $o(1)o(2)=-1$. Following \cite{LWZ23}, we define in $\tUi(\widehat{\mathfrak{sl}}_3,\tau)$ the {\em real $v$-root vectors}
	\begin{align}
		\label{eq:Bik}
		B_{i,k} =B_{k\delta+\alpha_i} :=\big(o(i)\TT_{\bome}\big)^{-k}(B_i), \quad
		\text{ for } k \in \Z, i \in \{1,2\}.
	\end{align}
	Denote, for $i=1,2$ and $k\in \Z$,
	\begin{equation}\label{Dn}
		D_{i,k} :=-[B_{\tau i},B_{i,k }]_{v^{-1}}-[B_{i,k+1},B_{\tau i,-1}]_{v^{-1}}.
	\end{equation}
	Set $\TH_{i,0} =\frac{1}{v-v^{-1}}$.
	Define the {\em imaginary $v$-root vectors} $\TH_{i,m}$, for $m\ge 1$, inductively:
	\begin{align}
		\label{TH}
		\TH_{i,1} & = -o(i)\big(\big[ B_i,[ B_{\tau i},B_0 ]_v\big]_{v^{2}}- v B_0 \bK_i\big),
		\\
		\label{eq:TH12}
		\TH_{i,2} &= -vD_{i,0} C \bK_{\tau i}^{-1}+v\TH_{i,0}C -\TH_{\tau i,0}C\bK_{\tau i}^{-1}\bK_i,
		\\
		\label{THn}
		\TH_{i,m} &=v\TH_{i,m-2} C -vD_{i,m-2}C \bK_{\tau i}^{-1},
		\quad \text{ for } m\ge 3.
	\end{align}
	For convenience, we set  $\TH_{i,m}=0$ for $m<0$.

\begin{proposition}[{\cite[Theorem 5.5]{LWZ23}}]
		\label{prop:Dr1}
There is an algebra isomorphism ${\Phi}: \tUiD (\widehat{\mathfrak{sl}}_3,\tau) \rightarrow\tUi(\widehat{\mathfrak{sl}}_3,\tau)$, which sends
\begin{align}
			\label{eq:isom}
	B_{i,l}\mapsto B_{i,l}, \quad \Theta_{i,m} \mapsto \Theta_{i,m},
			\quad
	\K_i\mapsto \K_i, \quad C\mapsto C,
	\quad \text{ for } m\ge 1,\; l\in \Z,\; i \in \{1,2\}.
\end{align}
The inverse ${\Phi}^{-1}$ sends
\begin{align*}
	\K_0\mapsto & -v^{-1} C \K_1^{-1}\K_2^{-1},
			\quad
	\K_i\mapsto  \K_i,
			\quad
	B_i\mapsto   B_{i,0},
	\quad \text{ for }i \in\{1,2\},
			\\
	B_0\mapsto
	& o(1) v^{-1} \big(\TH_{1,1}-v [B_1,B_{2,-1}]_{v^{-1}}  C\K_2^{-1}\big)\K_1^{-1}.
	\end{align*}
\end{proposition}

	%
	%%%%%%%%%%%%%%%%%%
	\section{Affine rank one subalgebras of affine iquantum groups}
	\label{sec:subalgebras}
	
	In this section, we construct new affine rank one subalgebras of affine iquantum groups $\tUi$ for each $i\in \I_0$ by establishing embeddings from the 2 different types of iquantum groups of affine rank one to $\tUi$.

\subsection{Definitions of affine rank one subalgebras}
	Define a sign function
	\[
	o(\cdot): \II \longrightarrow \{\pm 1\}
	\]
	such that $o(i) o(j)=-1$ whenever $c_{ij} <0$ (there are exactly 2 such functions).
	Inspecting the Satake diagram \eqref{pic:AIII2r}, we see that the values of $o(i)o(\tau i)=-1$ are independent of $i\in \I_0$. 
    
    \begin{lemma} \label{lem:TKK}
    We have $ \K_i^{-1} \TT_{\bome_i}^{-1}(\K_i)=v^{2r-1}\K_\delta$, for all $i\in \I_0$. 
    \end{lemma}

    \begin{proof}
    We prove by induction on $i$. 
        By using \eqref{eq:Tomega1}, 	\begin{align*}
			-\K_1^{-1}\TT_{\bome_1}^{-1}(\K_1)
			&=-\K^{-1}\TT_1^{-1}\cdots \TT_{r-1}^{-1} \TT_r^{-1}\TT_{r-1}^{-1}\cdots \TT_0^{-1}(\K_1)
			\\
			&=-v^{2r-1}\K_0\K_1\cdots\K_{2r}=-v^{2r-1}\K_\delta.
		\end{align*}
Assume $\K_i^{-1} \TT_{\bome_i}^{-1}(\K_i)=v^{2r-1}\K_\delta$, for $1\leq i\leq r-1$. By using 
Lemma \ref{lem:Tomeij}, we have 
\begin{align*}
    \TT_{\bome_{i+1}}^{-1}(\K_{i+1})&=\TT_i\TT_{\bome_i}^{-1}\TT_i\TT_{\bome_i}^{-1}\TT_{\bome_{i-1}}^{-1}(\K_{i+1})
    =\TT_i\TT_{\bome_i}^{-1}\TT_i(\K_{i+1})
    \\
    &=-v\TT_i\TT_{\bome_i}^{-1}(\K_i\K_{i+1})
   = -v\TT_i\TT_{\bome_i}^{-1}(\K_i)\TT_i\TT_{\bome_i}^{-1}(\K_{i+1})
    \\
    &=-v\TT_i(v^{2r-1}\K_\delta \K_i)\TT_i(\K_{i+1})
   -v^{2r-2}\K_i^{-1}\K_\delta(-v)(\K_i\K_{i+1})
    \\
    &=v^{2r-1}\K_\delta\K_{i+1}.
\end{align*}
This proves the desired formula for $1\leq i\leq r$. 

For $r+1\leq i\leq 2r$, it holds by applying the involution $\widehat{\tau}$.        
\end{proof}

We set 
	\begin{align}
		\label{C}
		C:=o(i)o(\tau i)\,\K_i^{-1} \TT_{\bome_i}^{-1}(\K_i)=-v^{2r-1}\K_\delta,
	\end{align}
	which is independent of $i\in \I_0$ by Lemma \ref{lem:TKK}.
	For $i\in\II$, we denote
	\[
	\bome_{i}':=\bome_i\bs_i, 
	\]
	and we have 
	\begin{align*}
		\ell(\bome_i') =\ell(\bome_i)-\ell(\bs_i), \qquad \text{and} \qquad \TT_{\bome_{i}'}=\TT_{\bome_i} \TT_i^{-1}. 
	\end{align*}
	
	\begin{definition} \label{def:Uii}
		For any $i\in\I_0$ such that $i\notin \{r,r+1\}$, we define $\tUi_{[i]}$ to be the $\Q(v)$-subalgebra of $\tUi$ generated by $B_j$, $\K_j^{\pm1}$, $\TT_{\bome_j'}(B_j)$,  $C^{\pm1}$ for $j\in\{i,\tau i\}$.
	\end{definition}

We shall also define the subalgebra $\tUi_{[i]}$ of $\tUi$, for $i\in\{r,r+1\}$. 
To that end, we first describe some properties of $W_\tau$ and $\Br(W_\tau)$. Note that $\I_0 =\I^c \cup \{r,r+1\}$, where 
    \[
    \I^c=\{1,2,\ldots,r-1,r+2,\ldots, 2r\}.
    \]
	
	Define $\bth_r$ to be the longest element in $W_{\I^c_\tau}:=\langle \bs_i\mid 1\le i <r \rangle$ (viewed as a subgroup of $W_\tau$), with a reduced expression given by 
    \[
    \bth_r=\bs_{r-1}(\bs_{r-2}\bs_{r-1})\cdots(\bs_{1}\cdots \bs_{r-1}).
    \]     
	Viewing $W_\tau$ (and $W_{\I^c_\tau}$) as a subgroup of $W$ (and $W_{\I^c}$), we can identify $\bth_r$ with the longest element in the parabolic subgroup $W_{\I^c}$ of $W$. We shall need the root vector associated to $\de-\alpha_r-\alpha_{\tau r}$:
	\begin{align}
		\label{def:Btheta}
		\TT_{\bth_r}^{-1} (B_0)
		=(\TT_{r-1}^{-1}\cdots\TT_{2}^{-1}\TT_{1}^{-1})\cdots(\TT_{r-1}^{r-1}\TT_{r-2}^{-1})\TT_{r-1}^{-1}(B_0)
		=\TT_{r-1}^{-1}\cdots\TT_{2}^{-1}\TT_1^{-1}(B_0).
	\end{align}
    We can now give a crucial definition of a new affine rank one subalgebra.
	\begin{definition} \label{def:Uir}
		We define $\tUi_{[r]} = \tUi_{[r+1]}$ to be the $\Q(v)$-subalgebra of $\tUi$ generated by $B_r, B_{r+1}$, $\K_r^{\pm 1}$, $\K_{r+1}^{\pm 1}$, $\TT_{\bth_r}(\K_0)^{\pm 1}$ and $\TT_{\bth_r}^{-1} (B_0)$. 
		%(Compare Definition~\ref{def:Uii}.)
	\end{definition}

	\subsection{Affine rank one subalgebra for $c_{r,\tau r}=-1$}
	
	To distinguish notations for $\tUi$, we shall adopt the dotted notation for the generators for the quasi-split affine iquantum group $\tUi(\widehat{\mathfrak{sl}}_3,\tau)$ of type AIII$_2^{(\tau)}$: $\dot B_i, \dot \K_i$, for $i=\{0,1,2\}$, whose Drinfeld prepsentation is given in Proposition \ref{prop:Dr1}. Accordingly we denote the relative braid group symmetries on $\tUi(\widehat{\mathfrak{sl}}_3,\tau)$ by $\dot \TT_w$. Recall the subalgebra $\tUi_{[r]}$ of $\tUi$ from Definition~\ref{def:Uir}. 
	
\begin{proposition}
		\label{prop:rank1iso-sl3}
	There exists a $\Q(v)$-algebra isomorphism  $\aleph_r: \tUi(\widehat{\mathfrak{sl}}_3,\tau) \longrightarrow \tUi_{[r]}$, which sends $\dot B_1 \mapsto B_r, \dot B_2 \mapsto B_{\tau r}, \dot B_0\mapsto  \TT_{\bth_r}^{-1} (B_0),\dot\K_1 \mapsto \K_r, \dot\K_2\mapsto \K_{\tau r}, \dot\K_0 \mapsto \TT_{\bth_r}(\K_0)$. In particular, for $j\in\{r, \tau r\}$, we have
	\begin{align}\label{TB00i}
		&\TT_{\bth_r}^{-1} (B_0)^2 B_j -[2] \TT_{\bth_r}^{-1} (B_0) B_j \TT_{\bth_r}^{-1} (B_0)+B_j \TT_{\bth_r}^{-1} (B_0)^2=-v^{-1}\TT_{\bth_r}(\K_0) B_j,
			\\
			&\TT_{\bth_r}^{-1} (B_0) B_j^2 -[2]  \TT_{\bth_r}^{-1} (B_0) B_j \TT_{\bth_r}^{-1} (B_0)+B_j^2 \TT_{\bth_r}^{-1} (B_0)=0. \label{TB0ii}
	\end{align}
\end{proposition}
	  
\begin{proof} 
Assume that $r\ge 2$ as the case $r=1$ is trivial. We only consider the case for $j=r$ as the other case for $j=r+1$ follows by symmetry. 
        
Let us prove \eqref{TB0ii}, which can be reformulated as
\begin{align*}
	S_r(\TT_{\bth_r}^{-1} (B_0))=0,
	\qquad \text{ where } S_r(x) :=\big[B_r,[B_r,x]_v\big]_{v^{-1}}.
\end{align*}
Since $\TT_k(B_0)=B_0$ for any $k\neq 1,2r$, we have
\begin{align*}
\TT_{\bth_r}^{-1} (B_0)
=\TT_{r-1}^{-1}\cdots \TT_2^{-1}\Big(\big[B_{2r},[B_1,B_0]_v\big]_v-v \K_{2r} B_0 \Big),
\end{align*}
and in addition, 
\begin{align*}
&\TT_{r-1}^{-1}\cdots \TT_2^{-1}(B_0)=B_0,\\
&\TT_{r-1}^{-1}\cdots \TT_2^{-1}(B_1)=\bigg[\Big[\cdots\big[ [B_{r-1},B_{r-2} ]_v,B_{r-3}\big]_v\cdots,B_2\Big]_v ,B_1\bigg]_v,\\
&\TT_{r-1}^{-1}\cdots \TT_2^{-1}(B_{2r})= \bigg[\Big[\cdots\big[ [B_{r+2}, B_{r+3}]_v,B_{r+4}\big]_v\cdots,B_{2r-1}\Big]_v ,B_{2r}\bigg]_v.
\end{align*}
Note that $B_r$ commutes with $\TT_{r-1}^{-1}\cdots \TT_2^{-1}(B_{2r})$ and $B_0$, as $B_r$ commutes with all the relevant $B_i$ in the above expression for $\TT_{r-1}^{-1}\cdots \TT_2^{-1}(B_{2r})$. For the same reason, we also have
\[
S_r\big(\TT_{r-1}^{-1}\cdots \TT_2^{-1}(B_1)\big)=\bigg[\Big[\cdots\big[ [S_r(B_{r-1}), 	B_{r-2}]_v,B_{r-3}\big]_v\cdots,B_2\Big]_v ,B_1\bigg]_v=0.
\]
where we used $S_r(B_{r-1})=0$; see the defining relation \eqref{eq:S6}. By the above computations, we have
\begin{align*}
	&S_r\Big( \TT_{r-1}^{-1}\cdots \TT_2^{-1}\big[B_{2r},[B_1,B_0]_v\big]_v\Big)
			\\
	&= \big[\TT_{r-1}^{-1}\cdots \TT_2^{-1}(B_{2r}),[S_r\big(\TT_{r-1}^{-1}\cdots \TT_2^{-1}(B_1)\big),B_0]_v\big]_v
			=0.
\end{align*}
It remains to show that $\TT_{r-1}^{-1}\cdots \TT_2^{-1}\big( \K_{2r} B_0 \big)$ is annihilated by $S_r(\cdot)$. By definition, we have $\TT_{r-1}\cdots \TT_2(\K_{2r})=(-v)^{r-2}\K_{2r}\K_{2r-1}\cdots \K_{r+2}$. Hence, we have
\begin{align*}
	S_r\Big(\TT_{r-1}^{-1}\cdots \TT_2^{-1}\big( \K_{2r} B_0 \big)\Big) &=(-v)^{r} 	\K_{2r}\K_{2r-1}\cdots \K_{r+2}\big[ B_r, [B_r,B_0]\big]_ {v^{-2}}=0.
\end{align*} 
Therefore, we have proved \eqref{TB0ii}.
		
By entirely similar arguments we prove the following equivalent version of \eqref{TB00i} with $j=r$:
\[
B_0^2 \TT_{\bth_r}(B_r )-[2]B_0 \TT_{\bth_r}(B_r) B_0 +\TT_{\bth_r}(B_r) B_0^2=-v^{-1}\K_0 \TT_{\bth_r}(B_r).
\]
The detail is skipped. 

The surjectivity of $\aleph_r$ is clear since all generators of $\tUi_{[i]}$ are in the image. It remains to prove the injectivity of $\aleph_r$ using a filtration argument. Recall from \eqref{eq:filt1} and \eqref{eq:filter} the natural filtrations on $\tUi$ (and $\tUi(\widehat{\mathfrak{sl}}_3)$, respectively) such that the associated graded are given by
\begin{align*} 
\mathrm{gr}\tUi \cong \U^-\otimes \bQ(v)[\K_i^{\pm1} | i\in \I],
    \qquad
\mathrm{gr}\tUi(\widehat{\mathfrak{sl}}_3)\cong \U(\widehat{\mathfrak{sl}}_3)^-\otimes \bQ(v)[\dot{\K}_i^{\pm1}\mid i=0,1,2].
\end{align*}
The map $\aleph_r$ is compatible with these two filtrations and then $\aleph_r$ induces a homomorphism $\aleph_r^{\text{gr}}: \U(\widehat{\mathfrak{sl}}_3)^-\rightarrow \U ^-$ on the associated graded algebras, which sends $F_1 \mapsto F_r, F_2\mapsto F_{r+1}, F_0 \mapsto T_{\bth_r}^{-1}(F_0)$. Here $T_w$ ($w\in W$) denote Lusztig's braid group symmetries; see \cite{Lus93}. We can use similar arguments as in \cite[Proposition 3.8]{Be94} to prove that this induced homomorphism $\aleph_r^{\text{gr}}$ is injective, and hence $\aleph_r$ is also injective.
\end{proof}
    
	\begin{remark}
	   The algebra isomorphism $\aleph_r: \tUi(\widehat{\mathfrak{sl}}_3,\tau) \longrightarrow \tUi_{[r]}$ has a variant that sends $\dot B_1 \mapsto B_r, \dot B_2 \mapsto B_{\tau r}, \dot B_0\mapsto  \TT_{\bth_r} (B_0),\dot\K_1 \mapsto \K_r, \dot\K_2\mapsto \K_{\tau r}, \dot\K_0 \mapsto \TT_{\bth_r}^{-1}(\K_0)$. However, the homomorphism $\aleph_r$ given in Proposition \ref{prop:rank1iso-sl3} is the only one compatible with the construction of root vectors below.
	\end{remark}

	\subsection{Affine rank one subalgebras for $c_{i,\tau i}=0$}

    Denote by $\ov{\U}(\widehat{\mathfrak{sl}}_2)$ the quotient algebra of $\tU(\widehat{\mathfrak{sl}}_2)$ modulo the ideal generated by the central element $\tK_\de-\tK'_\de$.
	To distinguish notations from those for $\tUi$, we shall adopt the dotted notation for the generators for $\ov{\U}(\widehat{\mathfrak{sl}}_2)$: $\dot F_a, \dot E_a, \dot K_a, \dot K_a'$, for $a\in \{0,1\}.$ Accordingly we denote by $\dot\TT_a$ Lusztig's braid group symmetry of $\ov{\U}(\widehat{\mathfrak{sl}}_2)$. Recall the subalgebra $\tUi_{[i]}$ of $\tUi$ from Definition \ref{def:Uii}. 
	
    \begin{proposition}[\text{\cite[Proposition 3.7]{LWZ24}}]
		\label{prop:rank1isoQG}
	For $i\in\II$ with $c_{i,\tau i}=0$,  there is a $\Q(v)$-algebra isomorphism  $\aleph_i: \ov{\U}(\widehat{\mathfrak{sl}}_2) \longrightarrow \tUi_{[i]}$, which sends 
    $\dot F_1 \mapsto B_i, \dot F_0 \mapsto \TT_{\bome_i'} (B_i), \dot E_1\mapsto B_{\tau i}, \dot E_0\mapsto \TT_{\bome_i'}(B_{\tau i}), \dot K_1 \mapsto K_i, \dot K'_1\mapsto \K_{\tau i}, \dot K_0 \mapsto \TT_{\bome_i'}(\K_i), \dot K'_0\mapsto \TT_{\bome_i'}( \K_{\tau i})$. 
    \end{proposition}

\subsection{Translation invariance of affine rank one subalgebras}
	
\begin{proposition} \label{prop:TiFix}
	Let $i, j\in \II$ be such that $ j\neq i,\tau i$. Then  $\TT_{\bome_i}(x)=x$, for all $x\in\tUi_{[j]}$.
\end{proposition}
	
	\begin{proof}
Clearly, all of $\K_j, \K_{\tau j},C$ are fixed by $\TT_{\bome_i}$. By Proposition~\ref{prop:fixB}, $\TT_{\bome_i}(B_j)=B_j$ and  $\TT_{\bome_i}(B_{\tau j})=B_{\tau j}$.
		
		For other generators, the proof is divided into the following two cases. 
		
		\underline{Case (1): $j\neq r,\tau r$}. Note that $\TT_j^{-1}(B_j)$ equals either $ -B_{\tau j}\K_{\tau j}^{-1} $ or $-vB_j\K_j^{-1}$, which implies that $\TT_j^{-1}(B_j)$ is fixed by $\TT_{\bome_i}$. By Lemma~\ref{lem:Tomeij}, we have
		\begin{align*}
			\TT_{\bome_i} \TT_{\bome_j'}(B_j)
			=\TT_{\bome_i} \TT_{\bome_j} \TT_j^{-1}(B_j)
			= \TT_{\bome_j} \TT_{\bome_i} \TT_j^{-1}(B_j)
			= \TT_{\bome_j} \TT_j^{-1}(B_j)
			= \TT_{\bome_j'}(B_j).
		\end{align*}
		Applying $\widehat{\tau}$ to the above formula, we obtain $\TT_{\bome_i} \TT_{\bome_j'}(B_{\tau j})= \TT_{\bome_j'}(B_{\tau j})$ as well.
				
\underline{Case (2): $j\in\{r,\tau r\}$}.  
It remains to prove by downward induction on $i$ that
\begin{align}
    \label{eq:omegaithetab0}
    \TT_{\bome_i}(\TT_{\bth_r}^{-1} (B_0))=\TT_{\bth_r}^{-1} (B_0), \quad \forall 1\leq i\leq r-1.
\end{align}
For $i=r-1$, by Corollary \ref{cor:Da}, we have 
\begin{align*}
    \TT_{\bome_{r-1}}(\TT_{\bth_r}^{-1} (B_0))&=(\TT_0\TT_1\cdots \TT_r)^{r-1}\TT_1\TT_2\cdots \TT_{r-1}(\TT_{r-1}^{-1}\cdots\TT_{2}^{-1}\TT_1^{-1}(B_0))
            \\
    &=(\TT_0\TT_1\cdots \TT_r)^{r-1}(B_0).
\end{align*}
Moreover, from Corollary \ref{cor:length}, we know $\bs_1\bs_2\cdots \bs_{r-1}(\bs_0\bs_1\cdots \bs_r)^{r-1}$ is a reduced expression by considering the reduced expression of $\bome_{r-1}^2$. A simple calculation shows 
$$
\bs_1\bs_2\cdots \bs_{r-1}(\bs_0\bs_1\cdots \bs_r)^{r-1}(\alpha_0)=\alpha_0.
$$ 
Hence, by Proposition~\ref{prop:fixB}, we have
\begin{align*}
    \TT_1\TT_2\cdots \TT_{r-1} (\TT_0\TT_1\cdots \TT_r)^{r-1}(B_0)=B_0,
\end{align*}
and then
\begin{align*}
    \TT_{\bome_{r-1}}(\TT_{\bth_r}^{-1} (B_0))=\TT_{\bth_r}^{-1} (B_0).
\end{align*}
            
Assume that \eqref{eq:omegaithetab0} holds for $i=k \ge 2$. By \eqref{TT:induction}, we have
\begin{align*}
    \TT_{\bome_{k-1}}(\TT_{\bth_r}^{-1} (B_0))=\TT_{\bome_{k}} \TT_{k-1}\cdots\TT_2 \TT_1 \TT_{\bome_1}^{-1} \TT_1\TT_2\cdots \TT_{k-1}(\TT_{\bth_r}^{-1} (B_0)).
\end{align*}
Using the inductive assumption, it is enough to prove
\begin{align*}
    \TT_{k-1}\cdots\TT_2 \TT_1 \TT_{\bome_1}^{-1} \TT_1\TT_2\cdots \TT_{k-1}(\TT_{\bth_r}^{-1} (B_0))=\TT_{\bth_r}^{-1} (B_0),
\end{align*}
which is equivalent to
\begin{align*}
    \TT_{k-1}\cdots \TT_2\TT_1\TT_{\bome_1}^{-1}\TT_1\TT_2\cdots \TT_{k-1}\TT_k^{-1}\cdots \TT_1^{-1}(B_0)=\TT_k^{-1}\cdots \TT_1^{-1}(B_0),
\end{align*}
which is in turn equivalent to (by using \eqref{eq:Tomega1})
\begin{align*}
 \TT_1\TT_2\cdots \TT_{k-1} \cdot \TT_{k+1}^{-1}\cdots \TT_{r}^{-1}   \TT_{r-1}^{-1}\cdots \TT_1^{-1}\TT_0^{-1} \TT_1\TT_2\cdots \TT_{k-1}\TT_k^{-1}\cdots \TT_2^{-1}\TT_1^{-1}(B_0)=B_0.
\end{align*} 
We can freely move $\TT_{k+1}^{-1}\cdots\TT_r^{-1}\TT_{r-1}^{-1}\cdots \TT_{k+1}^{-1}$ to the left of $\TT_1\TT_2\cdots\TT_{k-1}$ since they commute and note that  $\TT_{k+1}^{-1}\cdots\TT_r^{-1}\TT_{r-1}^{-1}\cdots \TT_{k+1}^{-1}$ fixes $B_0$. Hence the previous desired identity is equivalent to 
 \begin{align}
 \label{eqn:reduced-induction}
    \TT_1\TT_2\cdots \TT_{k-1}\TT_{k}^{-1}\cdots \TT_1^{-1}\TT_0^{-1} \TT_1\TT_2\cdots \TT_{k-1}\TT_k^{-1}\cdots \TT_2^{-1}\TT_1^{-1}(B_0)=B_0.
 \end{align}
Let us prove \eqref{eqn:reduced-induction} by induction on $k$ for $1\leq k\leq r-1$. 
If $k=1$, we have  
\begin{align*}
    \text{LHS}\eqref{eqn:reduced-induction}=\TT_1^{-1}\TT_0^{-1}\TT_1^{-1}(B_0)=B_0.
\end{align*}

Let $k\geq2$. Applying braid relations in the first and third equations below and moving $\TT_k^{-1}$ to the left as possible in the second equation, we have
\begin{align*}
   &\TT_1\TT_2\cdots \TT_{k-1}\TT_{k}^{-1}\cdots \TT_1^{-1}\TT_0^{-1} \TT_1\TT_2\cdots (\TT_{k-1}\TT_k^{-1}\TT_{k-1}^{-1})\cdots \TT_2^{-1}\TT_1^{-1}(B_0)
    \\
    &=\TT_1\TT_2\cdots \TT_{k-1}\TT_{k}^{-1}\cdots \TT_1^{-1}\TT_0^{-1} \TT_1\TT_2\cdots \TT_{k-2}(\TT_k^{-1}\TT_{k-1}^{-1}\TT_k)\TT_{k-2}^{-1}\cdots \TT_2^{-1}\TT_1^{-1}(B_0)
    \\
    &=\TT_1\TT_2\cdots \TT_{k-2}(\TT_{k-1}\TT_{k}^{-1}\TT_{k-1}^{-1}\TT_k^{-1})\TT_{k-2}^{-1}\cdots \TT_1^{-1}\TT_0^{-1} \TT_1\TT_2\cdots \TT_{k-2}
    \\
    &\quad\cdot\TT_{k-1}^{-1}\TT_{k-2}^{-1}\cdots \TT_2^{-1}\TT_1^{-1}(B_0)
    \\
    &=\TT_1\TT_2\cdots \TT_{k-2}(\TT_{k}^{-1}\TT_{k-1}^{-1})\TT_{k-2}^{-1}\cdots \TT_1^{-1}\TT_0^{-1} \TT_1\TT_2\cdots \TT_{k-2}\TT_{k-1}^{-1}\TT_{k-2}^{-1}\cdots \TT_2^{-1}\TT_1^{-1}(B_0)
    \\
    &=\TT_k^{-1}\big( \TT_1\TT_2\cdots \TT_{k-2}\TT_{k-1}^{-1}\TT_{k-2}^{-1}\cdots \TT_1^{-1}\TT_0^{-1} \TT_1\TT_2\cdots \TT_{k-2}\TT_{k-1}^{-1}\TT_{k-2}^{-1}\cdots \TT_2^{-1}\TT_1^{-1}(B_0)\big)
    \\
    &=\TT_k^{-1}(B_0)=B_0,
\end{align*}
  where the second last equation follows by the inductive assumption. 
So \eqref{eqn:reduced-induction} holds for $1\leq k\leq r-1$, and then  \eqref{eq:omegaithetab0} holds. 

  As $\TT_{\bome_i}$ fixes all the generators of $\tUi_{[j]}$, the proposition follows.
		\end{proof}

\subsection{Compatibility of ibraid group actions}

 Recall the isomorphisms $\aleph_i$ from Propositions \ref{prop:rank1iso-sl3} and \ref{prop:rank1isoQG}.
 
\begin{proposition}
		\label{prop:T1Ti}
	For $i\in\I_{0,\tau}$, we have
	\begin{align}
	\TT_i|_{\tUi_{[i]}} &= \aleph_i \circ {\dot \TT}_1 \circ \aleph_i^{-1},
			\label{eq:T1Ti}			\\
	\TT_{\bome_i}|_{\tUi_{[i]}} &= \aleph_i \circ {\dot \TT}_{\bome_1} \circ \aleph_i^{-1}.
			\label{eq:T1Ti2}
	\end{align}
\end{proposition}
	
\begin{proof}
The proof of the identity \eqref{eq:T1Ti} relies on the existence of the rank one formulas in \cite{WZ25} for $\TT_i$ acting on integrable modules which depends only on $B_i, B_{\tau i}, \K_i, \K_{\tau i}$, where $c_{i,\tau i}=0, -1$. The argument can be found in the proof of \cite[Proposition 3.10]{LWZ24}. 
    
The identity \eqref{eq:T1Ti2} is established just as in \cite{Be94} in case $c_{i,\tau i}=0$. It remains to establish the identity \eqref{eq:T1Ti2} when  $c_{i,\tau i}=-1$, i.e., $i\in\{r,r+1\}$. Without loss of generality, we assume $i=r$. 
		
We prove the identity \eqref{eq:T1Ti2} for $i=r$. Recall from Proposition~\ref{prop:rank1iso-sl3} that $\aleph_r: \tUi(\widehat{\mathfrak{sl}}_3,\tau) \rightarrow \tUi $  sends $\dot B_1 \mapsto B_r, \dot B_2 \mapsto B_{\tau r}, \dot B_0\mapsto \TT_{\bth_r}^{-1}(B_0), \dot\K_1 \mapsto \K_r, \dot\K_2\mapsto \K_{\tau r}, \dot\K_0 \mapsto \TT_{\bth_r}(\K_0)$.
		
Recall that $\dot \TT_{\bome_1}=\dot\TT_0 \dot \TT_1$, and hence we have
\begin{align*}
\dot \TT_{\bome_1}^{-1}(\dot B_0)=\dot \TT_1^{-1}(\dot B_0 \dot \K_0^{-1}),
    \qquad	\dot \TT_{\bome_1}^{-1}(\dot B_1)=\dot \TT_1^{-1}\big([\dot B_0,\dot B_1]_v\big).
\end{align*}
Hence, in order to prove the identity \eqref{eq:T1Ti2}, it suffices to verify the following two identities
\begin{align}
    \label{eq:T1Ti4}
\TT_{\bome_r}^{-1}(\TT_{\bth_r}^{-1}(B_0) )&= \TT_r^{-1}(\TT_{\bth_r}^{-1}(B_0\K_0^{-1})),
            \\
        \label{eq:T1Ti3}
\TT_{\bome_r}^{-1}(B_r)&=\TT_r^{-1}\big([ \TT_{\bth_r}^{-1}(B_0), B_r]_v\big).
\end{align}
Recall that $\bome_r=(\bs_0\cdots \bs_r)^r$ is a reduced expression and $\TT_{\bth_r}^{-1}(B_0)=\TT_{\bs_1\bs_2\cdots\bs_{r-1}}^{-1}(B_0)$. 
		
We prove \eqref{eq:T1Ti4}. Thanks to $\TT_0^{-1}(B_0)=B_0\K_0^{-1}$ and Corollary \ref{cor:Da}, the identity \eqref{eq:T1Ti4} is equivalent to
\begin{align*} 
	\TT_{\bs_0\bs_1\cdots\bs_{r}}^{-(r-1)} \TT_{r-1}^{-1}\cdots \TT_1^{-1}(B_0)=B_0.
\end{align*}
Using the braid relations, we have $\TT_{\bs_0\bs_1\cdots\bs_{r}}^{-1}\TT_a^{-1}=\TT_{a-1}^{-1}\TT_{\bs_0\bs_1\cdots\bs_{r}}^{-1}$ for $a\geq 2$, and then we have
\begin{align*}
	\TT_{\bs_0\bs_1\cdots\bs_{r}}^{-1} \TT_{a}^{-1}\cdots \TT_1^{-1}(B_0)
	&=\TT_{a-1}^{-1}\cdots \TT_1^{-1}\TT_{\bs_0\bs_1\cdots\bs_{r}}^{-1} \TT_1^{-1}(B_0)
			\\
	&=\TT_{a-1}^{-1}\cdots \TT_1^{-1}(B_0),
\end{align*}
where we used $\TT_1\TT_0\TT_1(B_0)=B_0$ and $\TT_i(B_0)=B_0,i>1$ in the second equality. Applying the above formula repeatedly for $r-1$ times gives us 
\begin{align*}
	\TT_{\bs_0\bs_1\cdots\bs_{r}}^{-(r-1)} \TT_{r-1}^{-1}\cdots \TT_1^{-1}(B_0)
	%&=\TT_{\bs_0\bs_1\cdots\bs_{r}}^{-(r-2)} \TT_{r-2}^{-1}\cdots \TT_1^{-1}(B_0)\\
	%&=\TT_{\bs_0\bs_1\cdots\bs_{r}}^{-(r-3)} \TT_{r-3}^{-1}\cdots \TT_1^{-1}(B_0)
	%=\cdots
	=B_0,
\end{align*}
whence \eqref{eq:T1Ti4}.
		
We prove \eqref{eq:T1Ti3}. Since $ \TT_{r-1}\TT_r \TT_{r-1} (B_r)=B_r$, we have
\begin{align*}
	\TT_{\bs_0\bs_1\cdots\bs_{r}}^{-1}(B_r)=\TT_{r-1}(B_r).
\end{align*}
By braid relations, we have $\TT_{a-1}\TT_{\bs_0\bs_1\cdots\bs_{r}}^{-1}=\TT_{\bs_0\bs_1\cdots\bs_{r}}^{-1}\TT_a$ for $a\geq 2$. Applying these two relations repeatedly, we obtain
\begin{align*}
	\TT_{\bome_r}^{-1}(B_r)&=  \TT_{\bs_0\bs_1\cdots\bs_{r}}^{-r}(B_r)= \TT_{\bs_0\bs_1\cdots\bs_{r}}^{-(r-1)}\TT_{r-1}(B_r)
=\TT_{\bs_0\bs_1\cdots\bs_{r}}^{-1} \TT_{1}\TT_{\bs_0\bs_1\cdots\bs_{r}}^{-(r-2)}(B_r)
			\\
	&=\TT_{\bs_0\bs_1\cdots\bs_{r}}^{-1} \TT_{1}\TT_2\TT_{\bs_0\bs_1\cdots\bs_{r}}^{-(r-3)}(B_r)
			=\cdots
	=\TT_{\bs_0\bs_1\cdots\bs_{r}}^{-1} \TT_{\bs_1\bs_2\cdots\bs_{r-1}}(B_r).
\end{align*}
Hence, the identity \eqref{eq:T1Ti3} is equivalent to the identity
\begin{align}\label{eq:T1Ti5}
	\TT_0^{-1}  \TT_{\bs_1\bs_2\cdots\bs_{r-1}}(B_r)=[B_0,\TT_{\bs_1\bs_2\cdots\bs_{r-1}}(B_r)]_v. 
\end{align}
Note that $\TT_a\TT_{a-1}(B_a)=B_{a-1}$  $(2\le a \le r-1)$ by Proposition~\ref{prop:fixB}. Using this identity repeatedly, the identity \eqref{eq:T1Ti5} can be proved as follows:
\begin{align*}
	\TT_0^{-1}  \TT_{\bs_1\bs_2\cdots\bs_{r-1}}(B_r)&= \big[B_r,\TT_0^{-1}  \TT_{\bs_1\bs_2\cdots\bs_{r-2}}(B_{r-1})\big]_v
	=\big[B_r,\TT_0^{-1}  \TT_{\bs_2\cdots\bs_{r-1}}^{-1}(B_{1})\big]_v
			\\
	&=\big[B_r,  [B_0, \TT_{\bs_2\cdots\bs_{r-1}}^{-1}(B_{1})]_v\big]_v
	=\big[B_r,  [B_0, \TT_{\bs_1\bs_2\cdots\bs_{r-2}}(B_{r-1})]_v\big]_v
			\\
	&=\big[B_0,  [B_r, \TT_{\bs_1\bs_2\cdots\bs_{r-2}}(B_{r-1})]_v\big]_v
	=\big[B_0,  \TT_{\bs_1\bs_2\cdots\bs_{r-1}}(B_{r})\big]_v.
\end{align*}
Therefore, \eqref{eq:T1Ti2} is proved.
\end{proof}

	%%%%%%%%
	%%%%%%%%
\section{A Drinfeld presentation of affine iquantum groups}
	\label{sec:Dpresentation}
	
In this section, we construct new $v$-root vectors in the quasi-split iquantum group $\tUi$ of type ${\rm AIII}_{2r}^{(\tau)}$, and then formulate a Drinfeld type presentation for $\tUi$. Proof for the presentation is partially given and will be completed in Section \ref{sec:relation}.

\subsection{$v$-Root vectors in higher ranks}
	\label{subsec:root vectors}
	Define a sign function
	\[
	o(\cdot): \II \longrightarrow \{\pm 1\}
	\]
	such that $o(j)=-o(i)$ whenever $c_{ij} <0$. (There are clearly exactly 2 such functions.) We define uiformly the following elements in $\tUi$ (called {\em real $v$-root vectors}), for $i\in \II$, $k\in \Z$: 
	\begin{align}
		\label{re root}
		B_{i,k} &= o(i)^{-k} \TT_{\bome_i}^{-k} (B_i).
	\end{align}
	In particular, we have $B_{i,0}=B_i$.
	
	Next we define case-by-case the imaginary $v$-root vectors $\TH_{i,n}$, for $i\in \II$, $n\ge 1$,  depending on whether $c_{i,\tau i}= 0$, or $-1$.

	\subsubsection{The case when $c_{i,\tau i}=0$}
	For $i\in\I_0$ with $c_{i,\tau i}=0$, we define
	\begin{align}
		\label{def:thetac=0}
		\Theta_{i,0}:=\frac{1}{v-v^{-1}},\qquad
		\Theta_{i,n}:=[B_{i,n},B_{\tau i}]\K_{ \tau i}^{-1}, \text{ for }n>0.
	\end{align}	
	
	By Proposition~\ref{prop:T1Ti}, $\Theta_{i,n}=o(i)^{n}[\TT_{\bome_i}^{-n}(B_i),B_{\tau i}]\K_{\tau i}^{-1}$ is identified with the $\aleph_i$-image of $[x_{1,-n}^-,x_{1,0}^+]K_1'^{-1} =\frac{\varphi_{1,-n}}{v-v^{-1}}$; see Proposition~\ref{prop:rank1isoQG} for $\aleph_i$.

	\subsubsection{The case when $c_{i,\tau i}=-1$}

    The following definitions were inspired by the constructions of $v$-root vectors in type AIII$_2^{(\tau)}$ \cite{LWZ23}.
	For $i\in\I_0$ and $n\in\Z$, we define \begin{equation}\label{GDn}
		D_{i,n }=-[B_{\tau i},B_{i,n }]_{v^{-1}}-[B_{i,n+1},B_{\tau i,-1}]_{v^{-1}}.
	\end{equation}
	Let $\TH_{i,0}=\frac{1}{v-v^{-1}}$. 
	Define $\TH_{i,n}$, for $n\ge 1$, recursively as follows:
	\begin{align}
		\label{GTH1}
		\TH_{i,1}&= -o(i)v\big(\big[ B_i,\TT_{\bome_i}(B_{\tau i})\big]_{v^{-1}} C\bK_{\tau i}^{-1}
		-\TT_{\bth_r}^{-1} (B_0) \bK_i\big),
		\\
		\label{eq:GTH13}
		\TH_{i,2}&=-vD_{i,0} C \bK_{\tau i}^{-1}+v^{-c_{i,\tau i}}\TH_{i,0}C
		-\TH_{\tau i,0}C\bK_{\tau i}^{-1}\bK_i,
		\\
		\label{GTHn1}
		\TH_{i,n}\bK_{\tau i}&= v^{-c_{i,\tau i}}\TH_{i,n-2}\bK_{\tau i}C -vD_{i,n-2}C.
	\end{align}
	We further set  $\TH_{i,n}=0$ for $n<0$. By definition, we have
    \begin{align}\label{GTH1'}
        \Theta_{i,1}=v[B_i,B_{\tau i,-1}]_{v^{-1}}C\K_{\tau i}^{-1}+o(i)v\TT_{\bth_r}^{-1} (B_0) \bK_i.
    \end{align}
	
\subsection{Translation invariance of imaginary root vectors}	

We can now formulate a key property shared by imaginary $v$-root vectors in all types.
	
	\begin{proposition}
		\label{prop:Theta-invariant Tomega}
		The following identity holds:
		\[
		\TT_{\bome_j}(\Theta_{i,n})=\Theta_{i,n},
		\]
		for all $i, j\in\II$ and $n \ge 1$.
	\end{proposition}
	
	\begin{proof}
		The statement for $j\neq i$ follows by Proposition~\ref{prop:TiFix}.
		
		The statement for $j=i$ is reduced to the affine rank one case by applying Proposition~\ref{prop:T1Ti}. The statement in all 2 types of affine rank one is known to hold: the case for $c_{i,\tau i}=0$ was proved in \cite{Da93}, while the case for $c_{i,\tau i}=-1$ was proved in \cite{LWZ23}.
	\end{proof}

	\subsection{A Drinfeld type presentation of quasi-split affine iquantum groups}

	\begin{definition} 
		\label{def:iDRA1}
		Let $\tUiD$ be the $\mathbb{C}(v)$-algebra generated by the elements $B_{i,l}$, $H_{i,m}$, $\bK_i^{\pm1}$, $C^{\pm1}$, where $i\in\II$, $l\in\Z$ and $m>0$, subject to the following relations \eqref{qsiDR0}--\eqref{qsiDR10}, for $i,j\in \II$, $m,n>0,$ and $ l,k,k_1,k_2\in \Z$:
		\begin{align}
			C^{\pm1} \text{ is central,} &\quad [\bK_i,\bK_j] =  [\bK_i,H_{j,n}] =0,
            \quad \bK_i B_{j,l}=v^{c_{\tau i,j}-c_{ij}} B_{j,l} \bK_i,
			\label{qsiDR0}
			\\
			\label{qsiDR1}
			&[H_{i,m},H_{j,n}] =0,
			\\
			\label{qsiDR2}
			&[H_{i,m},B_{j,l}] =\frac{[mc_{ij}]}{m} B_{j,l+m}-\frac{[mc_{\tau i,j}]}{m} B_{j,l-m}C^m,
			\\
			\label{qsiDR4}
			&[B_{i,k},B_{\tau i,l}]
			= \bK_{\tau i} C^l \TH_{i,k-l}- \bK_{i} C^k \Theta_{\tau i,l-k}
			\text{ if }c_{i,\tau i}=0,
			\\
			\notag 
			[B_{i,k},B_{\tau i,l+1}]_{v^{-c_{i,\tau i}}}&-v^{-c_{i,\tau i}}[B_{i,k+1},B_{\tau i,l}]_{v^{c_{i,\tau i}}}
			= -\Theta_{{\tau i},l-k+1}\bK_{i} C^k
			+v \Theta_{ {\tau i},l-k-1}\bK_{i} C^{k+1}
			\\
			&\qquad\qquad- \Theta_{i,k-l+1} \bK_{ \tau i} C^l
			+v \Theta_{i,k-l-1}\bK_{ \tau i} C^{l+1}, \; \text{ if }   { c_{i,\tau i}=-1},
			\label{qsiDR6}  \\
			&[B_{i,k}, B_{j,l}] =0,
			\qquad \text{ if }c_{ij}=0 \text{ and }\tau i\neq j, 
			\label{qsiDR7}
			\\
			\label{qsiDR3}
			&[B_{i,k},B_{j,l+1}]_{v^{-c_{ij}}} -v^{-c_{ij}}[B_{i,k+1},B_{j,l}]_{v^{c_{ij}}} =0, \qquad \text{ if } j\neq \tau i,
		\end{align}
		and the Serre relations
		\begin{align}
			&\bS_{i,j}(k_1,k_2|l)
			= 0, \text{ if }c_{i,j}=-1,  j \neq \tau i\neq i, \label{qsiDR9}
			\\
			\label{qsiDR10}
			&\bS_{i,\tau i}(k_1,k_2|l)
			=[2]\Sym_{k_1,k_2}\sum_{p\geq 0}v^{2p}
			\big[\TH_{\tau i,l-k_2-p}\K_i-v\TH_{\tau i,l-k_2-p-2}C\K_i, B_{i,k_1-p} \big]_{v^{-4p-1}}C^{k_2+p}
			\notag
			\\
			& \quad +v[2]\Sym_{k_1,k_2}\sum_{p\geq 0}v^{2p} \big[ B_{i,k_1+p+1},\TH_{i,k_2-l-p+1}\K_{\tau i}-v\TH_{i,k_2-l-p-1}C \K_{\tau i}\big]_{v^{-4p-3}} C^{l-1},\\
			&\hspace{3in} \text{ if }c_{i,\tau i}=-1.
			\notag
		\end{align}
		Here $H_{i,m}$ are related to $\Theta_{i,m}$ by the following equation for generating functions in $u$: 
		\begin{align}
			\label{exp}
			1+ \sum_{m\geq 1} (v-v^{-1})\Theta_{i,m} u^m  = \exp\Big( (v-v^{-1}) \sum_{m\geq 1} H_{i,m} u^m \Big),
		\end{align}
		and we have used the shorthand notation
		\begin{align}
			\label{eq:SS}
			\bS_{i,j}(k_1,k_2|l) : = \Sym_{k_1,k_2}\Big(B_{i,k_1}B_{i,k_2}B_{j,l}-[2]B_{i,k_1} B_{j,l} B_{i,k_2} + B_{j,l} B_{i,k_1}B_{i,k_2}\Big).
		\end{align}
	\end{definition}

    We first prepare a few lemmas for the proof of Theorem \ref{thm:Dr} on Drinfeld presentation. 
	\begin{lemma}\label{lem:thrloop}
		The element $\TT_{\theta_r}^{-1}(B_0)$ in \eqref{def:Btheta} can be written in terms of loop generators as follows
		\begin{align*}
			\TT_{\theta_r}^{-1}(B_0) =-o(i)v^{-1}\TH_{r,1}\K_r^{-1}
			+ \big[ B_r,\TT_{\bome_r}(B_{r+1})\big]_{v^{-1}} C\K_{r+1}^{-1}\K_r^{-1}.
		\end{align*}
	\end{lemma}
	
	\begin{proof}
		It follows from the definition \eqref{GTH1} of $\TH_{r,1}$.
	\end{proof}
	
	\begin{lemma}   
		For each $i\in \II$, there exists an algebra automorphism $\t_i$ on $\tUiD$ such that
		\begin{align*}
			\t_i (B_{i,r}) &=o(i) B_{i,r -1}, \quad \t_i (B_{\tau i,r}) =o(\tau i)B_{\tau i,r -1},\quad \t_i (H_{j,m}) =H_{j,m}, \quad \t_i(C) =C,
			\\
			\t_i (\K_i) &=o(i)o(\tau i) \K_i C^{-1}, \quad \t_i (\K_{\tau i}) =o(i)o(\tau i) \K_{\tau i} C^{-1}, \quad \t_i(B_{j,r})=B_{j,r},
		\end{align*}
		for all $r\in \Z, m\ge 1$, and $j\not \in \{i,\tau i\}$. Moreover, $\t_i=\t_{\tau i}$ and $\t_i \t_k =\t_k\t_i$ for all $i,k \in \II$.
	\end{lemma} 
	
	\begin{proof}
		Follows by inspection of the defining relations for $\tUiD$.
	\end{proof}

\begin{lemma}[{see \cite[Remark 4.17]{Da12}}]	
Let $a,b,c\in\tUi$. For any $u,w,t$, we have
\begin{align}
\label{eq:Jacobi}
			\big[a,[b,c]_{t/w}\big]_{uw}=\big[[a,b]_u,c\big]_t+u\big[ b,[a,c]_w\big]_{t/uw}.
		\end{align}
	\end{lemma}

Denote by $\bth_i$ the longest element in $W_{\{1,2,\dots,i-1\}}:=\langle \bs_j\mid 1\leq j<i-1\rangle$ for $1\leq i\leq r$. Then we have $\TT_{\bth_i}^{-1}(B_0)=\TT_{i-1}^{-1}\cdots \TT_{2}^{-1}\TT_1^{-1}(B_0)$. 
	\begin{lemma}
		\label{lem:induction-TB}
		For $1< i\leq r$, we have 
		\begin{align}
			\label{eq:Theta_iB0}
			\TT_{\bth_i}^{-1}(B_0)&=\big[B_{\tau (i-1)},[B_{i-1}, \TT_{\bth_{i-1}}^{-1}(B_0)]_v\big]_v-v\TT_{\bth_{i-1}}^{-1}(B_0)\K_{\tau(i-1)}
			\\
			\label{eq:Theta_iB0-1}
			&=\big[B_{i-1},[B_{\tau(i-1)}, \TT_{\bth_{i-1}}^{-1}(B_0)]_v\big]_v-v\TT_{\bth_{i-1}}^{-1}(B_0)\K_{i-1},
			\\
			\label{eq:Serre-Thetai-1}
			\big[B_{i-1},& [B_{i-1},\TT^{-1}_{\bth_{i-1}}(B_0)]_v\big]_{v^{-1}}=0. %			\qquad
			%\big[B_{\tau(i-1)},[B_{\tau(i-1)},\TT^{-1}_{\bth_{i-1}}(B_0)]_v\big]_{v^{-1}}=0,
		\end{align}
	\end{lemma}
	
	\begin{proof}
		We prove \eqref{eq:Theta_iB0} by induction. 
		For $i=2$, it follows from Theorem \ref{thm:Ti} (or rather its $\TT_i^{-1}$-counterpart).

        For $i>2$, by inductive hypothesis, we get 
        \begin{align*}
            \TT_{\bth_i}^{-1}(B_0)&=\TT_{i-1}^{-1}\big(\TT_{\bth_{i-1}}^{-1}(B_0)\big)
            \\
            &=\TT_{i-1}^{-1}\Big(\big[B_{\tau (i-2)},[B_{i-2}, \TT_{\bth_{i-2}}^{-1}(B_0)]_v\big]_v\Big)-v\TT_{i-1}^{-1}( \TT_{\bth_{i-2}}^{-1}(B_0)\K_{\tau(i-2)})
            \\
            &=\Big[[B_{\tau (i-1)},B_{\tau (i-2)}]_v,\big[[B_{i-1},B_{i-2}]_v, \TT_{\bth_{i-2}}^{-1}(B_0)\big]_v\Big]_v+v^2\TT_{\bth_{i-2}}^{-1}(B_0) \K_{\tau(i-2)}\K_{\tau(i-1)}
            \\
            &=\Big[[B_{\tau (i-1)},B_{\tau (i-2)}]_v,\big[B_{i-1},[B_{i-2}, \TT_{\bth_{i-2}}^{-1}(B_0)]_v\big]_v\Big]_v+v^2\TT_{\bth_{i-2}}^{-1}(B_0) \K_{\tau(i-2)}\K_{\tau(i-1)}
            \end{align*}
            where the last equality follows by \eqref{eq:Jacobi} and $[B_{i-1},\TT_{\theta_{i-2}}^{-1}(B_0)]=0$.
            Then 
            \begin{align*}
\TT_{\bth_i}^{-1}(B_0)            &=\bigg[B_{\tau (i-1)},\Big[B_{\tau (i-2)},\big[B_{i-1},[B_{i-2}, \TT_{\bth_{i-2}}^{-1}(B_0)]_v\big]_v\Big]_v\bigg]_v
            \\
            &\quad -v\bigg[B_{\tau (i-2)},\Big[B_{\tau (i-1)},\big[B_{i-1},[B_{i-2}, \TT_{\bth_{i-2}}^{-1}(B_0)]_v\big]_v\Big]\bigg] +v^2\TT_{\bth_{i-2}}^{-1}(B_0) \K_{\tau(i-2)}\K_{\tau(i-1)}
            \\
            &=\bigg[B_{\tau (i-1)},\Big[B_{i-1},\big[B_{\tau(i-2)},[B_{i-2}, \TT_{\bth_{i-2}}^{-1}(B_0)]_v\big]_v\Big]_v\bigg]_v
            \\
            &\quad-v\Big[B_{\tau(i-2)},\big[ [B_{\tau(i-1)},B_{i-1}], [B_{i-2}, \TT_{\bth_{i-2}}^{-1}(B_0)]_v\big]_v \Big]+v^2\TT_{\bth_{i-2}}^{-1}(B_0) \K_{\tau(i-2)}\K_{\tau(i-1)},
            \end{align*}
where the lat equality follows from \eqref{eq:Jacobi} and  
$$[B_{i-1},B_{\tau(i-2)}]=0, \qquad\big[B_{\tau(i-1)},[B_{i-2},\TT^{-1}_{\bth_{i-2}}(B_0)]_v\big]=0.$$ 
Note that 
\begin{align*}
    \bigg[B_{\tau (i-1)},\Big[B_{i-1},\big[B_{\tau(i-2)},[B_{i-2}, \TT_{\bth_{i-2}}^{-1}(B_0)]_v\big]_v\Big]_v\bigg]_v=\bigg[B_{\tau (i-1)},\Big[B_{i-1},\TT^{-1}_{\bth_{i-1}}(B_0)\Big]_v\bigg]_v,
\end{align*}
by using the inductive hypothesis, and 
\begin{align*}
    \bigg[B_{\tau (i-1)},\Big[B_{i-1},\TT_{\bth_{i-2}}^{-1}(B_0)\K_{\tau(i-2)}\Big]_v\bigg]_v= \bigg[B_{\tau (i-1)},\Big[B_{i-1},\TT_{\bth_{i-2}}^{-1}(B_0)\Big]\K_{\tau(i-2)}\bigg]_v=0.
\end{align*}
Therefore, 
            \begin{align*}
            \TT_{\bth_i}^{-1}(B_0)&=\bigg[B_{\tau (i-1)},\Big[B_{i-1},\TT^{-1}_{\bth_{i-1}}(B_0)\Big]_v\bigg]_v
            -v\Big[B_{\tau(i-2)},\big[ B_{i-2}, \TT_{\bth_{i-2}}^{-1}(B_0)\big]_v\Big]_v\K_{\tau(i-1)} 
            \\
            &\quad+v^2\TT_{\bth_{i-2}}^{-1}(B_0) \K_{\tau(i-2)}\K_{\tau(i-1)}
            \\
            &=\big[B_{\tau (i-1)},[B_{i-1}, \TT_{\bth_{i-1}}^{-1}(B_0)]_v\big]_v-v\TT_{\bth_{i-1}}^{-1}(B_0)\K_{\tau(i-1)},
        \end{align*}
        by using the inductive hypothesis again. The proof of \eqref{eq:Theta_iB0} is completed.

		The formula \eqref{eq:Theta_iB0-1} follows from \eqref{eq:Theta_iB0} since the braid group actions $\TT_j$ commutates with the involution $\widehat{\tau}$.  
		
		For \eqref{eq:Serre-Thetai-1}, we also prove by induction. If $i=2$, it is obvious by using \eqref{eq:S6}. 
        For $i\geq 2$, by using the inductive hypothesis, $[B_{i-1},B_{\tau(i-2)}]=0$, and $[B_{i-1},\TT_{\theta_{i-2}}^{-1}(B_0)]=0$, we have
		\begin{align*}
		    \big[B_{i-1}, [B_{i-1},\TT^{-1}_{\bth_{i-1}}(B_0)]_v\big]_{v^{-1}}&=\bigg[B_{i-1},\Big[B_{i-1},\big[B_{\tau(i-2)},[B_{i-2}, \TT_{\bth_{i-2}}^{-1}(B_0)]_v\big]_v\Big]_v\bigg]_v
            \\
            &=\bigg[B_{\tau(i-2)},\Big[B_{i-1},\big[B_{i-1},[B_{i-2}, \TT_{\bth_{i-2}}^{-1}(B_0)]_v\big]_v\Big]_v\bigg]_v
            \\
            &=\bigg[B_{\tau(i-2)},\Big[B_{i-1},\big[[B_{i-1},B_{i-2}]_v, \TT_{\bth_{i-2}}^{-1}(B_0)\big]_v\Big]_v\bigg]_v
            \\
            &=\bigg[B_{\tau(i-2)},\Big[\big[B_{i-1},[B_{i-1},B_{i-2}]_v\big]_v, \TT_{\bth_{i-2}}^{-1}(B_0)\Big]_v\bigg]_v=0
		\end{align*}
        since $\big[B_{i-1},[B_{i-1},B_{i-2}]_v\big]_v=0$; see \eqref{eq:S6}. The proof is completed.
	\end{proof}

	We can now formulate the main result of this paper. Recall the root vectors $B_{i,k}, \TH_{i,m}$ in $\tUi$ from \S\ref{subsec:root vectors}.
	
	\begin{theorem}
		\label{thm:Dr}
		There is a $\Q(v)$-algebra isomorphism ${\Phi}: {}\tUiD \longrightarrow\tUi$, which sends
		\begin{align}
			\label{def:Phi}
			B_{i,k}\mapsto B_{i,k}, \quad \Theta_{i,m} \mapsto \Theta_{i,m},
			\quad
			\K_i\mapsto \K_i, \quad C\mapsto o(i)o(\tau i)\K_i^{-1}\TT_{\bome_i}^{-1}(\K_i),
		\end{align}
		for $i\in \II, m\ge 1,$ and $k\in \Z$.
	\end{theorem}
	
	\begin{proof}
		We assume that $\Phi\colon \tUiD \rightarrow\tUi$ is an algebra homomorphism for now, postponing its proof to Section \ref{sec:relation}. 
		
		We first show that $\Phi: \tUiD \rightarrow\tUi$ is surjective. All generators for $\tUi$ except $B_0$ are clearly in the image of $\Phi$, and so it remains to show that $B_0\in \Im(\Phi)$.  
		
		We shall prove by downward induction on $i$, for $1\le i\le r$, that $\TT_{\bth_{i}}^{-1}(B_0)\in\Im(\Phi)$. The base case $\TT_{\bth_r}^{-1}(B_0)\in \text{Im} (\Phi)$ holds by Lemma~\ref{lem:thrloop}. Assume that $\TT_{\bth_{i}}^{-1}(B_0)\in\Im(\Phi)$, for $2\le i\le r$. By Lemma \ref{lem:induction-TB}, we have
		\begin{align*}
			&\big[B_{i-1},	\TT_{\bth_i}^{-1}(B_0)\big]_{v^{-1}}
			\\
			&=\Big[ B_{i-1}, \big[B_{\tau (i-1)},[B_{i-1}, \TT_{\bth_{i-1}}^{-1}(B_0)]_v\big]_v\Big]_{v^{-1}}-v\big[B_{i-1},\TT_{\bth_{i-1}}^{-1}(B_0)\K_{\tau(i-1)}\big]_{v^{-1}}
			\\
			&=\big[[B_{i-1},B_{\tau(i-1)}],[B_{i-1}, \TT_{\bth_{i-1}}^{-1}(B_0)]_v
		\big]+\Big[ B_{\tau(i-1)}, \big[B_{i-1}, [B_{i-1},\TT_{\bth_{i-1}}^{-1}(B_0)]_v\big]_{v^{-1}}\Big]_v
		\\
		&\quad -v\big[B_{i-1},\TT_{\bth_{i-1}}^{-1}(B_0)\big]_{v}\K_{\tau(i-1)}	
		\\
		&=\big[\frac{\K_{\tau(i-1)}-\K_{i-1}}{v-v^{-1}},[B_{i-1}, \TT_{\bth_{i-1}}^{-1}(B_0)]_v
		\big]-v\big[B_{i-1},\TT_{\bth_{i-1}}^{-1}(B_0)\big]_{v}\K_{\tau(i-1)}	
		\\
		&=v^{-1} \big[B_{i-1},\TT_{\bth_{i-1}}^{-1}(B_0)\big]_{v}\K_{i-1}.
		\end{align*}
	Hence, $\big[B_{i-1},\TT_{\bth_{i-1}}^{-1}(B_0)\big]_{v}$ and then $\big[B_{\tau(i-1)},[B_{i-1},\TT_{\bth_{i-1}}^{-1}(B_0)]_v\big]_v$ lie in $\Im(\Phi)$. 
   We can rewrite 
	\begin{align*}
		&\big[B_{\tau(i-1)},[B_{i-1},\TT_{\bth_{i-1}}^{-1}(B_0)]_v\big]_v
		\\
		&=\big[[B_{\tau(i-1)},B_{i-1}],\TT_{\bth_{i-1}}^{-1}(B_0)\big]_{v^2}+\big[B_{i-1},[B_{\tau(i-1)},\TT_{\bth_{i-1}}^{-1}(B_0)]_v\big]_v
		\\
		&=-v(\K_{i-1}-\K_{\tau(i-1)})\TT_{\bth_{i-1}}^{-1}(B_0)+\TT_{\bth_i}^{-1}(B_0)+v\TT_{\bth_{i-1}}^{-1}(B_0)\K_{i-1}
		\\
		&=v\K_{\tau(i-1)} \TT_{\bth_{i-1}}^{-1}(B_0)+\TT_{\bth_i}^{-1}(B_0).
	\end{align*}
Thus $\TT_{\bth_{i-1}}^{-1}(B_0)\in\Im(\Phi)$ since $\big[B_{\tau(i-1)},[B_{i-1},\TT_{\bth_{i-1}}^{-1}(B_0)]_v\big]_v$ and $\TT_{\bth_i}^{-1}(B_0)$ lie in $\Im(\Phi)$.
	
Therefore, $B_0=\TT_{\bth_{1}}^{-1}(B_0)\in\Im(\Phi)$, and $\Phi$ is surjective.		
	
		The injectivity of $\Phi$ is proved by the same argument as in \cite[Proof of Theorem~ 3.13]{LW21b} by passing to the associated graded, where $\Phi^{\text{gr}}$ becomes a well-known isomorphism \cite{Be94, Da15} for the Drinfeld presentation of half the affine quantum group. We skip the details.
	\end{proof}

	\begin{lemma}
		\label{lem:relationsreform}
		The following equivalences hold:
		
		(1) The identity \eqref{qsiDR1} is equivalent to the identity
		\begin{align}\label{qsiDR1reform}
			[\TH_{i,m},\TH_{j,n}]=0,\quad \forall i,j \in \II, \text{ and }m,n\geq1.
		\end{align}
		(2) The identity \eqref{qsiDR2} is equivalent to the identity	\begin{align}\label{qsiDR2reform}
			[\TH_{i,m},& B_{j,k}]+v^{c_{i,j}-c_{\tau i,j}}[\TH_{i,m-2},B_{j,k}]_{v^{2(c_{\tau i,j}-c_{i,j})}}C
			\\\notag
			&-v^{c_{i,j}}[\TH_{i,m-1},B_{j,k+1}]_{v^{-2c_{ i,j}}}- v^{-c_{\tau i, j}}[\TH_{i,m-1},B_{j,k-1}]_{v^{2c_{\tau i,j}}}C
			=0,
		\end{align}
		for any $i,j \in \II$, $m\ge 1$ and $k\in\Z$.
	\end{lemma}
	
	\begin{proof}
		The first statement is obvious, and the second one follows from \cite[Lemma~ 4.8]{LWZ24}.
	\end{proof}

	\subsection{Drinfeld presentation via generating functions}
	
	Introduce the generating functions
	\begin{align}\label{eq:Genfun}
		\begin{cases}
			\bB_{i}(z) =\sum_{k\in\Z} B_{i,k}z^{k},
			\\
			\bTH_{i}(z)  =1+ \sum_{m > 0}(v-v^{-1})\Theta_{i,m}z^{m},
			\\
			\bH_i(u)=\sum_{m\geq 1} H_{i,m} u^m,
			\\
			\bDel(z)=\sum_{k\in\Z}  C^k z^k.
		\end{cases}
	\end{align}
	The equation \eqref{exp} can be reformulated in terms of generating functions as 
	\begin{align}
		\label{expz}
		\bTH_{i}(z)  = \exp \big((v-v^{-1}) \bH_i(z)\big).
	\end{align}
	Introduce the following notation
	\begin{align*}
		& \bS_{i,j}(w_1,w_2|z)
		\\
		&=\Sym_{w_1,w_2} \big( \bB_{j}(z)\bB_{i}(w_{1})\bB_{i}(w_2) -[2] \bB_{i}(w_{1})\bB_{j}(z)\bB_{i}(w_{2}) +\bB_{i}(w_1)\bB_{i}(w_{2})\bB_{j}(z) \big).
	\end{align*}
	We can rewrite the defining relations for $\tUiD$ via generating functions in \eqref{eq:Genfun}, and hence reformulate Theorem~\ref{thm:Dr} as follows.
	
\begin{theorem} \label{thm:DrGF}
	$\tUi$ is generated by the elements $B_{i,l}$, $\Theta_{i,m}$, $\bK_i^{\pm1}$, $C^{\pm1}$, where $i\in\II$, $l\in\Z$ and $m>0$, subject to the following relations, for $i,j\in \II$:
	\begin{align}
		C \text{ is central, } & \quad [\bK_i,\bK_j] =  [\bK_i,\bTH_{j}(w)] =0,\quad \bK_i\bB_j(w)=v^{c_{\tau i,j}-c_{ij}} \bB_j(w) \bK_i,
			\label{Rel1GF} \\
		\bTH_i(z) \bTH_j(w) &=\bTH_j(w) \bTH_i(z),
			\\
		\bB_j(w)  \bTH_i(z)
		&= \frac{1 -v^{c_{ij}}zw^{-1}}{1 -v^{-c_{ij}}zw^{-1}} \cdot \frac{1 -v^{-c_{\tau i,j}}zw C}{1 -v^{c_{\tau i,j}} zw C}
		\bTH_i(z) \bB_j(w), 
			\\
		[\bB_i(z),\bB_{\tau i}(w)] &=\frac{\bDel(zw)}{v-v^{-1}} \big(\K_{\tau i}\bTH_i(z)-\K_i\bTH_{\tau i}(w)\big),
		\qquad \text{ if } c_{i,\tau i}=0,
			\\
			\label{qsiA1DRG5}
		(v^{-1}z-w) \bB_i(z)& \bB_{\tau i} (w)  +(v^{-1}w-z) \bB_{\tau i}(w) \bB_i(z)
			\\
			%%%3
		= \frac{\bDel(zw) }{1-v^2}&
		\big((z -vw) \K_{{i}} \bTH_{\tau i} (w) + (w -vz) \K_{\tau i} \bTH_i(z) \big), \qquad \text{ if } c_{i,\tau i}=-1,
			\notag
			\\
		[\bB_i(w), \bB_j(z)] &=0, \qquad\text{ if } c_{ij}=0, j\neq \tau i, 
			\label{BB0} \\
		(v^{c_{ij}}z -w) & \bB_i(z) \bB_j(w)  +(v^{c_{ij}}w-z) \bB_j(w) \bB_i(z)=0, \qquad \text{ if }j\neq \tau i,
			\label{BB1}
		\end{align}
and the Serre relations
	\begin{align}
	&\bS_{i,j}(w_1,w_2|z) =0,\qquad \text{ if } c_{ij}=-1,j\neq \tau i\neq i,
			\\ \notag
	& \mathbb{S}_{i,\tau i}(w_1,w_2|z)= \frac{1}{v-v^{-1}}\Big(-v^{-1}[2]\Sym_{w_1,w_2}\bDel(w_2 z) \frac{1-v w_2^{-1} z}{1-v^{-2} w_1 w_2^{-1}}\bB_i(w_1)\bTH_{\tau i}(z)\bK_i
			\\\notag
	&\qquad\qquad\qquad \quad  +[2]\Sym_{w_1,w_2}\bDel(w_2 z) \frac{1-vw_2^{-1}z}{1-v^2 w_1 w_2^{-1}}\bTH_{\tau i}(z)\bK_i\bB_i(w_1)
			\\\notag
	&\qquad\qquad\qquad\quad +v[2]\Sym_{w_1,w_2}\bDel(w_2 z)\frac{w_1^{-1}z-v w_1^{-1} w_2}{1-v^2 w_1^{-1} w_2}\bB_i(w_1)\bTH_{i}(w_2)\bK_{\tau i}
			\\
	&\qquad\qquad\qquad\quad+v^{-2}[2]\Sym_{w_1,w_2}\bDel(w_2 z) \frac{v w_1^{-1}w_2 - w_1^{-1}z}{1-v^{-2}w_1^{-1}w_2} \bTH_{i}(w_2)\bK_{\tau i}\bB_i(w_1)\Big), \label{qsiA1DRG6}
			\\\notag
	&\hspace{8cm} \text{ if } c_{i,\tau i}=-1.
	\end{align}
\end{theorem}
	
\begin{proof}
	The equivalences between \eqref{qsiDR6} and \eqref{qsiA1DRG5} as well as between \eqref{qsiDR10} and \eqref{qsiA1DRG6} are established in \cite[Theorem 5.7]{LWZ23}. The equivalences of other relations are established in \cite[Theorem~ 4.6]{LWZ24}.
\end{proof}

	%%%%%%%%
	%%%%%%%%
\section{Verification of Drinfeld type new relations}
	\label{sec:relation}
	
In this section, we prove that ${\Phi}: \tUiD \rightarrow\tUi$ defined by \eqref{def:Phi} is a homomorphism, completing the proof of Theorem \ref{thm:Dr}. 

\subsection{Relation \eqref{qsiDR7}--\eqref{qsiDR3}}
	
\begin{proposition}
		\label{prop:qsiDR7}
	Assume $c_{ij}=0$, for $i,j \in \II$ such that $\tau i\neq j$. Then
	$[B_{i,k},B_{j,l}]=0,$ for all $k,l\in\Z$.
\end{proposition}
	
\begin{proof}
	The identity for $k=l=0$, i.e., $[B_{i},B_{j}]=0,$ is the defining relation \eqref{eq:S1} for $\tUi$. The identity for general $k,l$ follows by applying $\TT_{\bome_i}^{-k} \TT_{\bome_j}^{-l}$ to the above identity and using Lemma~\ref{lem:Tomeij} and Proposition~\ref{prop:TiFix}.
\end{proof}

\begin{lemma} [\text{cf. \cite[Lemma 3.3]{Be94}}]
		\label{lem:TXij}
For $j\neq \tau i\in\II$ such that $c_{ij}=-1$, denote
\begin{align*}
		X_{ij}:= B_jB_i-vB_iB_j.
\end{align*}
Then we have
$\TT_{\bome_i}(X_{ji})=\TT_{\bome_j}(X_{ij}).$
\end{lemma}
	
\begin{proof}
Observe that $c_{ij}=-1$ implies that either $\ov{c}_{ij}=-1$ or $\ov{c}_{ji}=-1$.  Without loss of generality, assume that $\ov{c}_{ji}=-1$. In particular, $j\neq r,r+1$.
According to Theorem~\ref{thm:Ti}, if $\ov{c}_{ji}=-1$, we have
\begin{align*}
\TT_j^{-1}(B_i)=B_jB_i-vB_iB_j=X_{ij},
			\qquad
\TT_j(B_i)=B_iB_j-vB_jB_i=X_{ji}.
\end{align*} 
Thus by Lemma \ref{lem:Tomeij}(2), we have
\begin{align*}
\TT_{\bome_j}(X_{ij})&=\TT_{\bome_j} \TT_j^{-1}(B_i)= \TT_j\TT_{\bome_j}^{-1}\TT_{\bome_i}(B_i)
=\TT_{\bome_i}\TT_j(B_i)=\TT_{\bome_i}(X_{ji}).
\end{align*}
\end{proof}
	
Now we are ready to establish the relation \eqref{qsiDR3}.
\begin{proposition}
		\label{prop:iDR3a}
We have $[B_{i,k}, B_{j,l+1}]_{v^{-c_{ij}}}  -v^{-c_{ij}} [B_{i,k+1}, B_{j,l}]_{v^{c_{ij}}}=0$, for $j\neq \tau i \in \II$ and $k, l \in \Z.$
\end{proposition}
	
\begin{proof}
For $j=i$, it follows by transporting the corresponding relations in $\tU(\widehat{\mathfrak{sl}}_2)$ and $\tUi(\widehat{\mathfrak{sl}}_3)$ by using Proposition \ref{prop:rank1isoQG} and Proposition \ref{prop:rank1iso-sl3}.
		
It remains to consider the case $i\neq j\neq \tau i$. If $c_{ij}=0$, then the identity in the proposition follows directly by \eqref{qsiDR7}, which has been proved in Proposition \ref{prop:qsiDR7}.
		
Assume $c_{ij}=-1$. Note that
$v[B_{i,k+1},B_{j}]_{v^{-1}}= - o(i)^{k+1} \TT_{\bome_i}^{-(k+1)}(X_{ij}).$
By using Lemma \ref{lem:TXij}, we have
\begin{align*}
	[B_{i,k},B_{j,1}]_{v}&= B_{i,k}B_{j,1}-vB_{j,1}B_{i,k}
			\\
	& = o(i)^k o(j) \TT_{\bome_i}^{-k} \TT_{\bome_j}^{-1}(X_{ji})
	= - o(i)^{k+1} \TT_{\bome_i}^{-k} \TT_{\bome_i}^{-1} (X_{ij})
			\\
	&= - o(i)^{k+1} \TT_{\bome_i}^{-(k+1)}(X_{ij})=v[B_{i,k+1},B_{j}]_{v^{-1}}.
\end{align*}
Hence we have obtained an identity
$[B_{i,k},B_{j,1}]_{v}-v[B_{i,k+1},B_{j}]_{v^{-1}}=0.$
The identity in the proposition follows by applying $\TT_{\bome_j}^{-l}$ to this identity.
\end{proof}

\subsection{Relations \eqref{qsiDR1}--\eqref{qsiDR2} for $j\in\{i,\tau i\}$, and \eqref{qsiDR4}--\eqref{qsiDR6}}
		
	\begin{proposition}
		\label{prop:iDR31}
		\eqref{qsiDR1}--\eqref{qsiDR2} for $j\in\{i,\tau i\}$, and \eqref{qsiDR4}--\eqref{qsiDR6} hold in $\tUi$.
	\end{proposition}
	
	\begin{proof}
		The current relations in $\tUi(\widehat{\mathfrak{sl}}_3)$ are given in Definition~\ref{def:iDR} and Proposition~\ref{prop:Dr1}. Using Propositions \ref{prop:rank1isoQG} and \ref{prop:rank1iso-sl3}, one can transport these (rank one) relations in $\tU(\widehat{\mathfrak{sl}}_2)$ and $\tUi(\widehat{\mathfrak{sl}}_3)$ to the higher rank case, and then the desired relations follow.
	\end{proof}

	\subsection{Relation \eqref{qsiDR2} for $c_{ij}=0=c_{\tau i,j}$}
	\label{subsec:verifydeg}
	By definition, we have 
	\begin{align}
    \Theta_{r,1}= v[B_r,B_{r+1,-1}]_{v^{-1}}C\K_{r+1}^{-1}-v\TT_{\theta_r}^{-1}(B_0)\K_r.
    \end{align}
	
	\begin{lemma}
		\label{lem:comb}
		For $j\in\II$, we have
		\begin{align*}
			\Theta_{j,n}=
			\begin{cases}
				v C \Theta_{j,n-2}-\K_j^{-1} \big( [B_{\tau j},B_{j,n}]_{v}+[B_{j,n-1},B_{\tau j,1}]_{v} \big), & \text{ if } n\ge 3,
				\\
				\frac{\K_j\K_{\tau j}^{-1}C-vC}{v-v^{-1}}  -\K_j^{-1} \big( [B_{\tau j},B_{j,2}]_{v}+[B_{j,1},B_{\tau j,1}]_{v} \big), & \text{ if } n=2,
				\\
				-\K_{\tau j}^{-1}[B_{\tau j}, B_{j,1}]_{v}+o(j)v\TT_{\bth_j}^{-1}(B_0)\K_j, & \text{ if } n= 1,
			\end{cases}
		\end{align*} if $c_{j,\tau j}=-1$;
		and
		\begin{align*}
			\Theta_{j,n}=[B_{j,n},B_{\tau j}]\K_{\tau j}^{-1},
		\end{align*}
		if $c_{j,\tau j}=0$.
		In particular, for any $n\ge 1$, the element $\Theta_{j,n}$  is a $\Q(v)[C^{\pm 1},\K_j^{\pm 1},\K_{\tau j}^{\pm1}]$-linear combination of $1$,  and $[B_{j,k}, B_{\tau j,l+1}]_{v^{-c_{j,\tau j}}} +[B_{\tau j,l}, B_{j,k+1}]_{v^{-c_{j,\tau j}}}$, (together with $\TT_{\bth_j}^{-1}(B_0)$ if $c_{j,\tau j}=-1$) for $l, k \in \Z$.
	\end{lemma}
	
	\begin{proof}
		The recursion formulas in the lemma are reformulations of \eqref{qsiDR4} with $k=n$ and $l=0$, and \eqref{qsiDR6} with $k=n-1$, $l=0$. The second statement follows by an induction on $n$ using the recursion formulas. (A precise linear combination can be written down, but will not be needed.)
	\end{proof}
	
	\begin{proposition}
		\label{prop:iDR2}
		Assume $c_{ij}=0=c_{i,\tau j}$, for $i,j\in \II$. Then, for $m\geq 1$ and $k \in\Z$, we have
		\begin{align*}
			[\Theta_{i,m},B_{j,k}] &=0 =[H_{i,m},B_{j,k}].
		\end{align*}
	\end{proposition}

\begin{proof}
	We shall only prove the first equality $[\Theta_{i,m},B_{j,k}]=0$; the second equality follows as $H_{i,n}$ can be expressed in terms of $\Theta_{i,m}$ for various $m$. 
        
    \underline{Case (1): $c_{i,\tau i}=0$}. By Lemma~\ref{lem:comb} (with index $j$ replaced by $i$), it suffices to check that $[B_{i,k}, B_{\tau i,l}]$ commutes with $B_{j,r}$ for all $k,l,r$. But this clearly follows by the commutative relations $[ B_{i,k},B_{j,r}]=0=[B_{\tau i,l},B_{j,r}]$, that is \eqref{qsiDR7}, which is   proved in Proposition \ref{prop:qsiDR7}.

	\underline{Case (2): $c_{i,\tau i}\neq0$}. This only happens for $i=r$ or $r+1$. We shall prove $[\TT_{\bth_r}^{-1} (B_0),B_{j}]=0$. 
		By the proof of Proposition \ref{prop:rank1iso-sl3},
		we have
		$\TT_{\bth_r}^{-1} (B_0)=\TT_{r-1}^{-1}\TT_{r-2}^{-1}\cdots \TT_2^{-1}\TT_{1}^{-1}(B_0).$
		So it is equivalent to proving that
		\begin{align*}
			[B_0,\TT_{1}\TT_{2}\cdots \TT_{r-2}\TT_{r-1}(B_j)]=0.
		\end{align*}
		This identity follows as 
        \[
        \TT_{1}\TT_{2}\cdots \TT_{r-2}\TT_{r-1}(B_j)=\TT_{1}\TT_{2}\cdots \TT_{j}\TT_{j+1}(B_j)=\TT_{1}\TT_{2}\cdots \TT_{j-1} (B_{j+1})=B_{j+1}, 
        \]
		where we have used  
		$\TT_{j}\TT_{j+1}(B_j)=B_{j+1}$.

        The same argument as in Case (1) shows that $B_j$ commutes with $[B_{i,k}, B_{\tau i,l+1}]_{v^{-c_{i,\tau i}}} +[B_{\tau i,l}, B_{i,k+1}]_{v^{-c_{i,\tau i}}}$. Hence by Lemma~\ref{lem:comb} (with index $j$ therein replaced by $i$), we have $[\Theta_{i,m},B_j]=0$. Applying $(o(j)\TT_{\bome_j})^{-k}$ to this identity gives $[\Theta_{i,m},B_{j,k}]=0$ since $\TT_{\bome_j}(\Theta_{i,m})=\Theta_{i,m}$ (see Proposition \ref{prop:Theta-invariant Tomega}).
	\end{proof}

	\subsection{Relation \eqref{qsiDR2}}
	
	The relation \eqref{qsiDR2} for $c_{ij}=0=c_{\tau i,j}$ has been verified in \S\ref{subsec:verifydeg}. 
	
	\begin{proposition}
		Assume $i\neq r,r+1$. 
		\begin{itemize}
			\item[(1)] If $c_{ij}=-1$ and $c_{\tau i,j}=0$, then  we have  $[H_{i,m}, B_{j,l}] = -\frac{[m]}{m} B_{j,m+l}$.
			\item[(2)] If $c_{ij}=0$ and $c_{\tau i,j}=-1$, then  we have  $[H_{i,m}, B_{j,l}] = \frac{[m]}{m} B_{j,l-m}C^m$.
		\end{itemize}
	\end{proposition}

	\begin{proof}
		The proof is the same as \cite[Proposition 5.6]{LWZ24}, hence omitted here.
	\end{proof}

	We next consider \eqref{qsiDR2} for the remaining case $i=r$ or $r+1$. Due to the symmetry $\widehat{\tau}$, it suffices to consider the case $i=r$.

	\begin{lemma}
		\label{lem:qsiDR100}
		For any $l\in\Z$, we have 
		\begin{align}\label{eq:qsiDR100}
			[\Theta_{r,1},B_{r-1,l}]=-B_{r-1,l+1}.
		\end{align}
	\end{lemma}
	
	\begin{proof}
		%We only prove \eqref{eq:qsiDR100}. 
		Without loss of generality, assume $ o(r-1)=1$ and $l=0$ by using Proposition \ref{prop:Theta-invariant Tomega}. By definition, we have 
        \[
        \Theta_{r,1}= v[B_r,B_{r+1,-1}]_{v^{-1}}C\K_{r+1}^{-1}-v\TT_{\theta_r}^{-1}(B_0)\K_r. 
        \]
        We have
		\begin{align*}
			&\quad v\big[[B_r,B_{r+1,-1}]_{v^{-1}}C\K_{r+1}^{-1},B_{r-1}\big]\\
			&=v^2\big[[B_r,B_{r+1,-1}]_{v^{-1}}, B_{r-1}\big]_{v^{-1}}C\K_{r+1}^{-1}\\
			&=v \Big(v\big[B_r,[B_{r+1,-1}, B_{r-1}] \big]_{v^{-2}}-\big[B_{r+1,-1},[B_r, B_{r-1}]_{v^{-1}}\big]_{v }\Big)C\K_{r+1}^{-1}\\
			&=-v\big[B_{r+1,-1},[B_r, B_{r-1}]_{v^{-1}}\big]_{v }C\K_{r+1}^{-1}\\
			&=- \big[B_{r+1,-1},[B_{r,-1}, B_{r-1,1}]_{v }\big]_{v} C\K_{r+1}^{-1}
		\end{align*}
		where the last equality follows by applying \eqref{qsiDR3}. By definition, we have $B_{i,-1}=o(i)\TT_{\bome_i}(B_i)$ and $\TT_{\bome_i}(\K_i)=o(i)o(\tau i)\K_i C^{-1}$ for $i\in \I_0$. By Theorem~\ref{thm:Ti}, we have
		\begin{align*}
			-\big[B_{r+1,-1},[B_{r,-1}, B_{r-1,1}]_{v}\big]_{v}
			&=\TT_{\bome_r} \TT_{\bome_{r-1}}^{-1}\big( \TT_r^{-1}(B_{r-1})+ B_{r-1}\K_{r+1}\big)
			\\
			&=\TT_{\bome_r} \TT_{\bome_{r-1}}^{-1}  \TT_r^{-1}(B_{r-1}) - B_{r-1,1} \K_{r+1} C^{-1}.
		\end{align*}
		By Corollary \ref{cor:Da}, we have
		\begin{align*}
			\TT_{\bome_r} \TT_{\bome_{r-1}}^{-1}  \TT_r^{-1}(B_{r-1})&=\TT_{\bs_1\bs_2\cdots\bs_{r-1}}^{-1}\TT_{\bs_0\bs_1\cdots\bs_{r-1}}(B_{r-1})
			\\
			&=-\TT_{\bs_1\bs_2\cdots\bs_{r-1}}^{-1}\TT_{\bs_0\bs_1\cdots\bs_{r-2}}(\K_{r-1}^{-1}B_{r+2})
			\\
			&=-v^{-2}\TT_{\bs_1\bs_2\cdots\bs_{r-1}}^{-1}\TT_{\bs_0\bs_1\cdots\bs_{r-2}}(B_{r+2})\TT_{\bs_1\bs_2\cdots\bs_{r-1}}^{-1}\TT_{\bs_0\bs_1\cdots\bs_{r-2}}(\K_{r-1}^{-1})
			\\
			&=v \TT_{\bs_1\bs_2\cdots\bs_{r-1}}^{-1}\TT_{\bs_0\bs_1\cdots\bs_{r-2}}(B_{r+2}) C^{-1}\K_{r-1}\K_r\K_{r+1}.
		\end{align*}
		Putting together all the above computations, we have
		\begin{align}
			\label{eq:pfiDR2''}
			[\Theta_{r,1},B_{r-1}]
			=- B_{r-1,1}+v \TT_{\bs_1\bs_2\cdots\bs_{r-1}}^{-1}\TT_{\bs_0\bs_1\cdots\bs_{r-2}}(B_{r+2}) \K_{r-1}\K_r
			-v[\TT_{\theta_r}^{-1}(B_0)\K_r , B_{r-1}].
		\end{align}
              
		To finish proving \eqref{eq:qsiDR100}, it remains to show that the last two terms in \eqref{eq:pfiDR2''} cancel. Indeed, we have
		\begin{align*}
			[\TT_{\theta_r}^{-1}(B_0)\K_r,B_{r-1}]&=v[\TT_{r-1}^{-1}\cdots\TT_1^{-1}(B_0),B_{r-1}]_{v^{-1}}\K_r\\
			&=v\TT_{\bs_1\bs_2\cdots\bs_{r-1}}^{-1}[B_0,\TT_{\bs_1\bs_2\cdots\bs_{r-1}}(B_{r-1})]_{v^{-1}}\K_r\\
			&=-v^{-1}\TT_{\bs_1\bs_2\cdots\bs_{r-1}}^{-1}[B_0,\TT_{\bs_1\bs_2\cdots\bs_{r-2}}(B_{r+2}\K_{r-1}^{-1})]_{v^{-1}}\K_r\\
			&=-v \TT_{\bs_1\bs_2\cdots\bs_{r-1}}^{-1}[B_0,\TT_{\bs_1\bs_2\cdots\bs_{r-2}}(B_{r+2} )]_{v^{-1}}\K_{r-1}\K_r\\
			&=-v \TT_{\bs_1\bs_2\cdots\bs_{r-1}}^{-1}[B_0,\TT_{\bs_2\cdots\bs_{r-1}}^{-1}(B_{2r})]_{v^{-1}}\K_{r-1}\K_r\\
			&=\TT_{\bs_1\bs_2\cdots\bs_{r-1}}^{-1}\TT_{\bs_2\cdots\bs_{r-1}}^{-1}\TT_0(B_{2r})\K_{r-1}\K_r\\
			&=\TT_{\bs_1\bs_2\cdots\bs_{r-1}}^{-1}\TT_{\bs_0\bs_1\cdots\bs_{r-1}}^{-1} (B_{1})\K_{r-1}\K_r.
		\end{align*}
		
		It remains to show that
		\begin{align}
			\TT_{\bs_1\bs_2\cdots\bs_{r-1}}^{-1}\TT_{\bs_0\bs_1\cdots\bs_{r-2}}(B_{r+2})= \TT_{\bs_1\bs_2\cdots\bs_{r-1}}^{-1}\TT_{\bs_0\bs_1\cdots\bs_{r-1}}^{-1} (B_{1})
		\end{align}
		which is equivalent to
		\begin{align}
			\label{eq:pfiDR22}
			\TT_{\bs_0\bs_1\cdots\bs_{r-1}} \TT_{\bs_0\bs_1\cdots\bs_{r-2}} (B_{r+2})= B_1.
		\end{align}
		Since $\bs_0\bs_1\cdots\bs_{r-1}\bs_0\bs_1\cdots\bs_{r-2}$ is a reduced expression, \eqref{eq:pfiDR22} follows by Proposition~\ref{prop:fixB} and 
		$$\bs_{0}\bs_1\cdots \bs_{r-1}\bs_0\bs_1\cdots\bs_{r-2}(\alpha_{r+2})=\alpha_1.
        $$ 
        The proof is completed.
	\end{proof}
	
	\begin{lemma}
		\label{lem:pfiDR4}
		For any $l\in\Z$, we have 
		\begin{align}
			\label{eq:pfiDR4}
			[\Theta_{r,1},B_{r+2,l}]=B_{r+2,l-1}C.
		\end{align}
	\end{lemma}
	
\begin{proof}
	It suffices to prove that $[\Theta_{r,1},B_{r+2}]=B_{r+2,-1}C$, the special case of \eqref{eq:pfiDR4} when $l=0$ and $o(r-1)=1$. (The general case follows by applying $(o(r-1)\TT_{\bome_{r-1}})^{-l}$ to this special case.) 
    
    To that end, recall 
    \[
    \Theta_{r,1}= v[B_{r},B_{r+1,-1}]_{v^{-1}}C\K_{r+1 }^{-1}-v\TT_{\theta_r}^{-1}(B_0)\K_{r}. 
    \]
    We have
		\begin{align*}
			&\big[[B_{r},B_{r+1,-1}]_{v^{-1}}C\K_{r+1}^{-1},B_{r+2}\big]
			\\
			&=v^{-1}\big[[B_{r},B_{r+1,-1}]_{v^{-1}}, B_{r+2}\big]_{v}C\K_{r+1}^{-1}
			\\
			&= \Big(v^{-1}\big[B_{r},[B_{r+1,-1}, B_{r+2}]_v \big]_{v^{-1}}-v^{-2}\big[B_{r+1,-1},[B_{r}, B_{r+2}]\big]_{v^2 }\Big)C\K_{r+1}^{-1}
			\\
			&=v^{-1}\big[B_{r},[B_{r+1,-1}, B_{r+2}]_v \big]_{v^{-1}} C\K_{r+1}^{-1}
			\\
			&=\big[B_{r},[B_{r+1}, B_{r+2,-1}]_{v^{-1}}\big]_{v^{-1}} C\K_{r+1}^{-1},
		\end{align*}
		where the last equality follows by applying \eqref{qsiDR3}. Since $o(r+2)=-1$, we have $B_{r+2,-1}=-\TT_{\bome_{r-1}}(B_{r+2})$. By Theorem~\ref{thm:Ti}, we have
		\begin{align*}
			\big[B_{r},[B_{r+1},B_{r+2,-1}]_{v^{-1}}\big]_{v^{-1}}
			&=-v^{-2}\TT_{\bome_{r-1}}\big( \TT_r (B_{r+2})+ v B_{r+2}\K_{r+1 }\big)
			\\
			&=v^{-2} \TT_r (B_{r+2,-1})+ v^{-1} B_{r+2,-1}\K_{r+1}.
		\end{align*}
		
		On the other hand, we have
		\begin{align*}
			v[\TT_{\theta_r}^{-1}(B_0)\K_{r} , B_{r+2}]
			&=[\TT_{r-1}^{-1}\cdots\TT_1^{-1}(B_0) , B_{r+2}]_v\K_{r}
			\\
			&=\TT_{r-1}^{-1}\cdots\TT_1^{-1}[B_0 , \TT_{\bs_1\bs_2\cdots\bs_{r-1}}(B_{r+2})]_v\K_{r}
			\\
			&=-v^{-2}\TT_{r-1}^{-1}\cdots\TT_1^{-1}[B_0 , \TT_{\bs_1\bs_2\cdots\bs_{r-2}}(B_{r-1}\K_{r+2}^{-1})]_v\K_{r}
			\\
			&=-\TT_{r-1}^{-1}\cdots\TT_1^{-1}[B_0 , \TT_{\bs_1\bs_2\cdots\bs_{r-2}}(B_{r-1})]_v\K_{r+2} \K_{r}
			\\
			&=-\TT_{\bs_1\bs_2\cdots\bs_{r-1}}^{-1} [B_0 , \TT_{\bs_2\cdots\bs_{r-1}}^{-1}(B_{1})]_v\K_{r+2} \K_{r}
			\\
			&=-\TT_{\bs_1\bs_2\cdots\bs_{r-1}}^{-1}  \TT_{\bs_2\cdots\bs_{r-1}}^{-1}\TT_0^{-1} (B_{1}) \K_{r+2} \K_{r}.
		\end{align*}
	Collecting the above computations, we have
		\begin{align*}
			[\Theta_{r,1},B_{r+2}]
			= B_{r+2,-1}C +v^{-1}\TT_r (B_{r+2,-1}) C\K_{r+1}^{-1}
			+\TT_{\bs_1\bs_2\cdots\bs_{r-1}}^{-1}  \TT_{\bs_2\cdots\bs_{r-1}}^{-1}\TT_0^{-1} (B_{1}) \K_{r+2} \K_{r}.
		\end{align*}
		
		To prove that $[\Theta_{r,1},B_{r+2}]=B_{r+2,-1}C$, it remains to show that 
		\begin{align}\label{eq:pfiDR6}
			v^{-1}\TT_r (B_{r+2,-1}) C\K_{r+1}^{-1}
			+\TT_{\bs_1\bs_2\cdots\bs_{r-1}}^{-1}  \TT_{\bs_2\cdots\bs_{r-1}}^{-1}\TT_0^{-1} (B_{1}) \K_{r+2} \K_{r}=0.
		\end{align}
		Recall that $\TT_{\bome_r}=(\TT_0\TT_1\cdots \TT_r)^r$. 
        By \eqref{eq:Tomegar-1}, we have 
		\begin{align*}
			v^{-1}\TT_r (B_{r+2,-1}) C\K_{r+1}^{-1}&=-v^{-1}\TT_r\TT_{\bome_{r-1}}(B_{r+2})C\K_{r+1}^{-1}
			\\
			&=-v^{-1}\TT_r(\TT_0\TT_1\cdots \TT_r)^{r-1}\TT_1\TT_2\cdots \TT_{r-1}(B_{r+2})C\K_{r+1}^{-1}
            \\
            &=v^{-1} \TT_{\bs_1\bs_2\cdots \bs_{r-1}}^{-1}\TT_0^{-1}\TT_{\bome_r}\TT_1\TT_2\cdots \TT_{r-2}(\K_{r+2}^{-1} B_{r-1})C\K_{r+1}^{-1}
            \\
            &=v^{-3} \TT_{\bs_1\bs_2\cdots \bs_{r-1}}^{-1}\TT_0^{-1}\TT_1\TT_2\cdots \TT_{r-2}( B_{r-1}\K_{r+2}^{-1})C\K_{r+1}^{-1}
            \\
            &\overset{(*)}{=}- \TT_{\bs_1\bs_2\cdots \bs_{r-1}}^{-1}\TT_0^{-1}\TT_1\TT_2\cdots \TT_{r-2}( B_{r-1} )\K_{r+2}\K_{r} 
            \\
            &=- \TT_{\bs_1\bs_2\cdots \bs_{r-1}}^{-1}\TT_0^{-1}\TT_{\bs_2\cdots \bs_{r-1}}^{-1}( B_{1} )\K_{r+2}\K_{r},
		\end{align*}
        where the equality (*) follows by
        \[
        \TT_{\bs_1\bs_2\cdots \bs_{r-1}}^{-1}\TT_0^{-1}\TT_1\TT_2\cdots \TT_{r-2}(\K_{r+2}^{-1})
        = v^{-2r+4} \K_{\delta}^{-1} \K_{r+2} \K_{r+1}\K_r=-v^3 C^{-1}\K_{r+2} \K_{r+1}\K_r.
        \]
        Thus, \eqref{eq:pfiDR6} holds. The lemma is proved.
	\end{proof}

	\begin{proposition} For any $l\in\Z$, $m>0$, we have 
		\begin{align}
			\label{HBrr-1}
			[H_{r,m}, B_{r-1,l}] &= -\frac{[m]}{m} B_{r-1,l+m},
			\\
			\label{HBrr+1}
			[H_{r,m}, B_{r+2,l}] &= \frac{[m]}{m} B_{r+2,l-m}C^m.
		\end{align}
	\end{proposition}
	
	\begin{proof}
		Let us prove \eqref{HBrr-1}. By Lemma~\ref{lem:relationsreform}(2), the identity \eqref{HBrr-1} is equivalent to
        \begin{align*}
        [\Theta_{r,m},B_{r-1,l}]-v^{-1}[\Theta_{r,m-1},B_{r-1,l+1}]_{v^2}=
        [\Theta_{r,m-1},B_{r-1,l-1}]C-v^{-1}[\Theta_{r,m-2},B_{r-1,l}]_{v^2}C,
        \end{align*}
		which is implied by the following identity 
		\begin{align}
			\label{eq:HBequiv1}
			[\Theta_{r,m},B_{r-1,l}]=v^{-1}[\Theta_{r,m-1},B_{r-1,l+1}]_{v^2}, \quad \text{ for }m>0 \text{ and }l\in\Z.
		\end{align}
		The application of $\TT_{\bome_{r-1}}$ allows the reduction of proving \eqref{eq:HBequiv1} to its special case for $l=0$.
		For $m=1$, the identity \eqref{eq:HBequiv1} is proved in Lemma \ref{lem:qsiDR100}. 
		For $m=2$, by \eqref{eq:GTH13}, we have
		\begin{align*}
			&\quad[\Theta_{r,2},B_{r-1}]
			\\
			&=\Big(v^2\big[ [B_{r+1},B_r]_{v^{-1}},B_{r-1}\big]_{v^{-1}}C+v^2\big[[B_{r,1},B_{r+1,-1}]_{v^{-1}},B_{r-1}\big]_{v^{-1}}C-vB_{r-1}C\K_r\Big)\K_{r+1}^{-1}
			\\
			&=v^2\big[B_{r+1},[B_r,B_{r-1}]_{v^{-1}}\big]_{v^{-1}}C\K_{r+1}^{-1}-v\big[B_{r+1,-1},[B_{r,1},B_{r-1}]_{v^{-1}}\big]_{v}C\K_{r+1}^{-1}-vB_{r-1}C\K_{r+1}^{-1}\K_r\\
			&=v\big[B_{r+1},[B_{r,-1},B_{r-1,1}]_{v}\big]_{v^{-1}}C\K_{r+1}^{-1}-\big[B_{r+1,-1},[B_{r},B_{r-1,1}]_{v}\big]_{v}C\K_{r+1}^{-1}-vB_{r-1}C\K_{r+1}^{-1}\K_r,
		\end{align*}
		where we used \eqref{qsiDR7} and \eqref{qsiDR3} that has been proved in Propositions \ref{prop:qsiDR7} and \ref{prop:iDR3a}.	
		Applying \eqref{GTH1'}, we have 
		\begin{align*}
			&\quad [\Theta_{r,2},B_{r-1}]
			\\
			&=v\big[B_{r-1,1},[B_{r,-1},B_{r+1}]_{v}\big]_{v^{-1}}C\K_{r+1}^{-1}-v^2\big[B_{r-1,1},[B_{r},B_{r+1,-1}]_{v^{-1}}\big]_{v^{-1}}C\K_{r+1}^{-1}
            \\
            &\qquad-vB_{r-1}C\K_{r+1}^{-1}\K_r
			\\
			&=-\big[[B_{r,-1},B_{r+1}]_{v},B_{r-1,1}\big]_{v}C\K_{r+1}^{-1}+v\big[[B_r,B_{r+1,-1}]_{v^{-1}},B_{r-1,1}\big]_vC\K_{r+1}^{-1}-vB_{r-1}C\K_{r+1}^{-1}\K_r\\
			&=[\Theta_{r,1}C^{-1}\K_{r+1},B_{r-1,1}]_vC\K_{r+1}^{-1}-vo(r)[\TT^{-1}_{\bth_{r}}(B_0)\K_rC^{-1}\K_{r+1},B_{r-1,1}]_vC\K_{r+1}^{-1}
			\\
			&\quad +[\Theta_{r+1,1}C^{-1}\K_{r},B_{r-1,1}]_vC\K_{r+1}^{-1}+vo(r)[\TT^{-1}_{\bth_r}(B_0)\K_{r+1}C^{-1}\K_{ r},B_{r-1,1}]_vC\K_{r+1}^{-1}
            \\
            &\quad -vB_{r-1}C\K_{r+1}^{-1}\K_r
			\\
			&=v^{-1}[\Theta_{r,1},B_{r-1,1}]_{v^2}+v[\Theta_{r+1,1},B_{r-1,1}]\K_{r+1}^{-1}\K_r-vB_{r-1}C\K_{r+1}^{-1}\K_r
			\\
			&=v^{-1}[\Theta_{r,1},B_{r-1,1}]_{v^2}.
		\end{align*}
		Here the last equality holds since $[\Theta_{r+1,1},B_{r-1,1}]=B_{r-1}C$; see Lemma \ref{lem:pfiDR4}.
		
		For $m\geq3$, by \eqref{GTHn1}, \eqref{qsiDR7} and \eqref{qsiDR3}, we have
		\begin{align*}
			&[\Theta_{r,m},B_{r-1}]-v[\Theta_{r,m-2}C,B_{r-1}]
			\\
			&=v\big[[B_{r+1,1},B_{r,m-1}]_{v^{-1}}\K_{r+1}^{-1},B_{r-1}\big]+v\big[[B_{r,m},B_{r+1}]_{v^{-1}}\K_{r+1}^{-1},B_{r-1}\big]
			\\
			&=v^2\big[[B_{r+1,1},B_{r,m-1}]_{v^{-1}},B_{r-1}\big]_{v^{-1}}\K_{r+1}^{-1}+v^2\big[[B_{r,m},B_{r+1}]_{v^{-1}},B_{r-1}\big]_{v^{-1}}\K_{r+1}^{-1}
			\\
			&=v^2\big[B_{r+1,1},[B_{r,m-1},B_{r-1}]_{v^{-1}}\big]_{v^{-1}}\K_{r+1}^{-1}-v\big[B_{r+1 },[B_{r,m},B_{r-1}]_{v^{-1}} \big]_{v}\K_{r+1}^{-1}
			\\
			&=-\big[[B_{r,m-2},B_{r-1,1}]_{v},B_{r+1,1}\big]_{v}\K_{r+1}^{-1}+v\big[[B_{r,m-1},B_{r-1,1}]_{v},  B_{r+1}\big]_{v^{-1}}\K_{r+1}^{-1}.
		\end{align*}
		By using  \eqref{eq:Jacobi} again, we have
		\begin{align}\notag
			&\quad [\Theta_{r,m},B_{r-1}]-v[\Theta_{r,m-2}C,B_{r-1}]
			\\\label{eq:ThBinduction}
			&=v\big[B_{r-1,1},[B_{r,m-2},B_{r+1,1}]_v\big]_{v^{-1}}\K_{r+1}^{-1}-v^2\big[B_{r-1,1},[B_{r,m-1},B_{r+1}]_{v^{-1}}\big]_{v^{-1}}\K_{r+1}^{-1}.
		\end{align}
		
		If $m=3$, using \eqref{eq:GTH13}, we rewrite \eqref{eq:ThBinduction} as follows
		\begin{align*}
			&\quad [\Theta_{r,3},B_{r-1}]-v[\Theta_{r,1}C,B_{r-1}]
			\\	
			&=[B_{r-1,1},-v\Theta_{r,2}\K_{r+1}+\frac{v^2C\K_{r+1}}{v-v^{-1}}-\frac{vC\K_r}{v-v^{-1}}]_{v^{-1}}\K_{r+1}^{-1}
			\\
			&=-v[B_{r-1,1},\Theta_{r,2}]_{v^{-2}}+vB_{r-1,1}C
			\\
			&=v^{-1}[\Theta_{r,2},B_{r-1,1}]_{v^2}+vB_{r-1,1}C.
		\end{align*}
		Together with Lemma \ref{lem:qsiDR100}, we obtain
		$[\Theta_{r,3},B_{r-1}]=v^{-1}[\Theta_{r,2},B_{r-1,1}]_{v^2}$, which proves \eqref{eq:HBequiv1}.
		
		If $m\geq4$, we rewrite \eqref{eq:ThBinduction} as follows
		\begin{align*}
			&\quad [\Theta_{r,m},B_{r-1}]-v[\Theta_{r,m-2}C,B_{r-1}]
			\\	
			&=-[B_{r-1,1},v\Theta_{r,m-1}\K_{r+1}-v^2\Theta_{r,m-3}C\K_{r+1}]_{v^{-1}} \K_{r+1}^{-1}
			\\
			&=-v[B_{r-1,1},\Theta_{r,m-1}-v\Theta_{r,m-3}C]_{v^{-2}}
			\\
            &=v^{-1}[\Theta_{r,m-1},B_{r-1,1}]_{v^{2}}-[\Theta_{r,m-3},B_{r-1,1}]C.
		\end{align*}
		By induction, we have
		$$[\Theta_{r,m},B_{r-1}]=v^{-1}[\Theta_{r,m-1},B_{r-1,1}]_{v^2},$$
        which proves the desired identity \eqref{eq:HBequiv1}.

        Now let us prove \eqref{HBrr+1}, in a way similar to \eqref{HBrr-1}. Note that \eqref{HBrr+1} is equivalent to 
        \begin{align}
        \label{HBrr+1-reform}
			[\TH_{r,m},B_{r+2,l}]= v[\TH_{r,m-1},B_{r+2,l-1}]_{v^{-2}}C,\quad \forall m>0,l\in\Z.
		\end{align}
        It is enough to prove \eqref{HBrr+1-reform} for $l=0$.

        For $m=1$, it is proved in Lemma \ref{lem:pfiDR4}. For $m=2$,  by \eqref{eq:GTH13} and \eqref{qsiDR7}--\eqref{qsiDR3}, we have
		\begin{align*}
			& [\Theta_{r,2},B_{r+2}]
			\\
			&=\big[ [B_{r+1},B_r]_{v^{-1}},B_{r+2}\big]_v C\K_{r+1}^{-1}+\big[[B_{r,1},B_{r+1,-1}]_{v^{-1}},B_{r+2}\big]_{v}C\K_{r+1}^{-1}+v^{-1}B_{r+2}C\K_{r+1}^{-1}\K_r
			\\
&=\big[[B_{r+1},B_{r+2}]_{v},B_r\big]_{v^{-1}}C\K_{r+1}^{-1}+\big[B_{r,1},[B_{r+1,-1},B_{r+2}]_{v}\big]_{v^{-1}}C\K_{r+1}^{-1}+v^{-1}B_{r+2}C\K_{r+1}^{-1}\K_r\\
			&=v\big[[B_{r+1,1},B_{r+2,-1}]_{v^{-1}},B_r\big]_{v^{-1}}C\K_{r+1}^{-1}+v\big[B_{r,1},[B_{r+1},B_{r+2,-1}]_{v^{-1}}\big]_{v^{-1}}C\K_{r+1}^{-1}\\
            &\quad+v^{-1}B_{r+2}C\K_{r+1}^{-1}\K_r.
		\end{align*}
		Applying \eqref{GTH1'} and Proposition~\ref{prop:Theta-invariant Tomega}, we have 
		\begin{align*}
			&\quad[\Theta_{r,2},B_{r+2}]\\
            &=v\big[[B_{r+1,1},B_{r}]_{v^{-1}},B_{r+2,-1}\big]_{v^{-1}}C\K_{r+1}^{-1}+v\big[[B_{r,1},B_{r+1}]_{v^{-1}},B_{r+2,-1}\big]_{v^{-1}}C\K_{r+1}^{-1}
            \\
            &\quad+v^{-1}B_{r+2}C\K_{r+1}^{-1}\K_r
			\\
			&=\big[\Theta_{r+1,1}\K_r,B_{r+2,-1}\big]_{v^{-1}}C\K_{r+1}^{-1}+o(r)v\big[\TT_{\bth_{r}}^{-1}(B_0)\K_r\K_{r+1},B_{r+2,-1}\big]_{v^{-1}}C\K_{r+1}^{-1}
            \\
    &\quad+\big[\Theta_{r,1}\K_{r+1},B_{r+2,-1}\big]_{v^{-1}}C\K_{r+1}^{-1}-o(r)v\big[\TT_{\bth_{r}}^{-1}(B_0)\K_{r}\K_{r+1},B_{r+2,-1}\big]_{v^{-1}}C\K_{r+1}^{-1}
    \\
    &\quad+v^{-1}B_{r+2}C\K_{r+1}^{-1}\K_r\\
			&=v^{-1}[\Theta_{r+1,1},B_{r+2,-1}]C\K_r\K_{r+1}^{-1}+v[\Theta_{r,1},B_{r+2,-1}]_{v^{-2}}C+v^{-1}B_{r+2}C\K_{r+1}^{-1}\K_r
			\\
			&=v[\Theta_{r,1},B_{r+2,-1}]_{v^{-2}}C.
		\end{align*}
		Here the last equality holds since $[\Theta_{r+1,1},B_{r+2,-1}]=-B_{r+2}$; see Lemma \ref{lem:qsiDR100}.
		
		For $m\geq3$, by \eqref{GTHn1}, \eqref{qsiDR7} and \eqref{qsiDR3}, we have
		\begin{align*}
			&[\Theta_{r,m},B_{r+2}]-v[\Theta_{r,m-2}C,B_{r+2}]
			\\
			&=v\big[[B_{r+1,1},B_{r,m-1}]_{v^{-1}}\K_{r+1}^{-1},B_{r+2}\big]+v\big[[B_{r,m},B_{r+1}]_{v^{-1}}\K_{r+1}^{-1},B_{r+2}\big]
			\\
			&=\big[[B_{r+1,1},B_{r,m-1}]_{v^{-1}},B_{r+2}\big]_{v}\K_{r+1}^{-1}+\big[[B_{r,m},B_{r+1}]_{v^{-1}},B_{r+2}\big]_{v}\K_{r+1}^{-1}
			\\
			&=\big[[B_{r+1,1},B_{r+2}]_v,B_{r,m-1}\big]_{v^{-1}}\K_{r+1}^{-1}+\big[B_{r,m},[B_{r+1 },B_{r+2}]_v \big]_{v^{-1}}\K_{r+1}^{-1}
			\\
			&=v\big[[B_{r+1,2},B_{r+2,-1}]_{v^{-1}},B_{r,m-1}\big]_{v^{-1}}\K_{r+1}^{-1}+v\big[B_{r,m},[B_{r+1,1},B_{r+2,-1}]_{v^{-1}}\big]_{v^{-1}}\K_{r+1}^{-1}.
		\end{align*}
		By using  \eqref{eq:Jacobi} again, we have
		\begin{align*}
			&[\Theta_{r,m},B_{r-1}]-v[\Theta_{r,m-2}C,B_{r-1}]
			\\
			&=v\big[[B_{r+1,2},B_{r,m-1}]_{v^{-1}},B_{r+2,-1}\big]_{v^{-1}}\K_{r+1}^{-1}+v\big[[B_{r,m},B_{r+1,1}]_{v^{-1}},B_{r+2,-1}\big]_{v^{-1}}\K_{r+1}^{-1}.
		\end{align*}
		
		If $m=3$, using \eqref{eq:GTH13}, we have
		\begin{align*}
			&[\Theta_{r,3},B_{r+2}]-v[\Theta_{r,1}C,B_{r+2}]
			\\	
			&=[\Theta_{r,2}\K_{r+1}-\frac{vC\K_{r+1}}{v-v^{-1}}+\frac{C\K_r}{v-v^{-1}},B_{r+2,-1}]_{v^{-1}}C\K_{r+1}^{-1}
			\\
			&=v[\Theta_{r,2},B_{r+2,-1}]_{v^{-2}}C-vB_{r+2,-1}C^2.
		\end{align*}
		Together with Lemma \ref{lem:qsiDR100}, we obtain
		$[\Theta_{r,3},B_{r-1}]=v[\Theta_{r,2},B_{r+2,-1}]_{v^{-2}}C$.
		
		If $m\geq4$, we have
		\begin{align*}
			&[\Theta_{r,m},B_{r+2}]-v[\Theta_{r,m-2}C,B_{r+2}]
			\\	
			&=[\Theta_{r,m-1}\K_{r+1}-v\Theta_{r,m-3}C\K_{r+1},B_{r+2,-1}]_{v^{-1}} C\K_{r+1}^{-1}
			\\
            &=v[\Theta_{r,m-1},B_{r+2,-1}]_{v^{-2}}C-v^2[\Theta_{r,m-3},B_{r+2,-1}]_{v^{-2}}C^2.
		\end{align*}
		By induction, we have
		$$[\Theta_{r,m},B_{r+2}]=v[\Theta_{r,m-1},B_{r+2,-1}]_{v^{-2}}C.$$
        
	\end{proof}

	\subsection{Relation \eqref{qsiDR1} for $j\notin\{r,r+1\}$}
	
	We shall derive the identity $[H_{i,m},H_{j,n}]=0$ in \eqref{qsiDR1}, for $j\neq r,r+1$, from the relations \eqref{qsiDR2}--\eqref{qsiDR6} (proved above).
	
	\begin{lemma}
		\label{lem:comm}
		Let $i,j \in \II$ such that $c_{j,\tau j}=0$. For any $l, k \in \Z$ and  $m\ge 1$,  we have
		\[
		\Big[ H_{i,m}, [B_{j,k},B_{\tau j,l+1}]_{v^{-c_{j,\tau j}}}+[B_{\tau j,l},B_{j,k+1}]_{v^{-c_{j,\tau j}}} \Big] =0.
		\]
	\end{lemma}
	
	\begin{proof}
		The proof is completely the same as \cite[Lemma 5.8]{LWZ24}, hence omitted here.
	\end{proof}

	\begin{proposition}
		\label{prop:iDR1}
		Relation \eqref{qsiDR1} holds for $j\notin\{i,\tau i\}$. 
	\end{proposition}
	
	\begin{proof}
		Without loss of generality, we assume $c_{j,\tau j}=0$. It follows by Lemma~\ref{lem:comb} and Lemma~\ref{lem:comm} that $[ H_{i,m}, \Theta_{j,a}] =0$, for all $m, a\ge 1.$ Since $H_{j,n}$ for any $n \ge 1$ is a linear combination of monomials in $\Theta_{j,a}$, for various $a\ge 1$ by \eqref{exp h}, we conclude that $[ H_{i,m}, H_{j,n}] =0$, whence \eqref{qsiDR1}.
	\end{proof}
	
	\subsection{Relations \eqref{qsiDR9}--\eqref{qsiDR10}}
	
	\begin{proposition}
		Relation \eqref{qsiDR10} holds in $\tUi$.
	\end{proposition}
	
	\begin{proof}
		It follows from Propositions \ref{prop:rank1iso-sl3}, \ref{prop:T1Ti} and \eqref{qsiA1DR6}.
	\end{proof}	
	
	In the remainder of this subsection, we shall verify \eqref{qsiDR9} in $\tUi$.
	
	We shall fix $i,j \in \II$ such that $c_{ij}=-1$ and $i\neq \tau j$ throughout this subsection.
	
	\begin{lemma}
		\label{lem:SSSa}
		Assume $i,j \in \II$ such that $c_{ij}=-1$ and $i\neq \tau j$. 
		For $k_1, k_2, l \in \Z$, we have
		\begin{align*}
			& \SS(k_1,k_2+1 |l) + \SS(k_1+1,k_2|l) -[2] \SS(k_1+1,k_2+1 |l-1)=0.
		\end{align*}
	\end{lemma}

	\begin{proof}
		The proof is the same as for \cite[Lemma 4.9]{LW21b}, and hence omitted here. It uses only the relations \eqref{qsiDR3}, which have been established above.
	\end{proof}

	\begin{lemma}  \label{lem:SSS}
		Assume $i,j \in \II$ such that $c_{ij}=-1$ and $i\neq \tau j$.	For $k_1, k_2, l \in \Z$, we have
		\begin{align*}
			& \SS(k_1,k_2+1 |l) + \SS(k_1+1,k_2|l) -[2] \SS(k_1,k_2|l+1)=0.
		\end{align*}
	\end{lemma}

	\begin{proof}
		The proof is the same as that of \cite[Lemma 4.13]{LW21b}, hence omitted here. It uses only the relations \eqref{qsiDR3}, which have been established above.	
	\end{proof}
	
	\begin{proposition}
		Relation \eqref{qsiDR9} holds in $\tUi$.
	\end{proposition}
	
	\begin{proof}
		It follows by using the same argument of \cite[\S5.1]{Z22} by  using Lemmas \ref{lem:SSSa}--\ref{lem:SSS}.
	\end{proof}	

This completes the verification that ${\Phi}: {}\tUiD \longrightarrow\tUi$ is a homomorphism and hence the proof of Theorem \ref{thm:Dr}. 
	%
	%

	%%%%%%%%%%%%%
	%%%%%%%%%%%%%

\end{document}